\documentclass[preprint,11pt]{elsarticle}
\usepackage[T1]{fontenc}
\usepackage{amssymb,amsmath,graphicx,color,amsthm}
\usepackage[capitalise]{cleveref}
\usepackage{xspace}
\usepackage[utf8]{inputenc}

\usepackage[labelformat=simple]{subcaption}

\newtheorem{theorem}{Theorem}
\newtheorem{lemma}{Lemma}
\newtheorem{corollary}{Corollary}

\newtheorem{conjecture}{Conjecture}
\newtheorem{problem}{Problem}

\def\figscale{0.9}

\newcommand{\set}[1]{\ensuremath{\left\{#1 \right\}}}
\newcommand{\chife}[0]{\chi_{\textrm{$\ell$-f}}^\prime}

\renewcommand{\int}{{\rm int}}
\newcommand{\ext}{{\rm ext}}
\newcommand{\fec}{$3$-FEC\xspace}
\newcommand{\chf}{\ensuremath{\mathrm{ch}_\mathrm{f}}}

\begin{document}

\begin{frontmatter}

\title{$3$-facial edge-coloring of plane graphs}

\author[im]{Mirko Hor\v{n}\'{a}k\fnref{ap}}
\ead{mirko.hornak@upjs.sk}

\author[ib,if]{Borut Lu\v{z}ar\fnref{arrs}}
\ead{borut.luzar@gmail.com}

\author[ib]{Kenny \v{S}torgel\fnref{mr,arrs}}
\ead{kennystorgel.research@gmail.com}

\fntext[ap]{Supported by the Slovak Research and Development Agency under the contract APVV--19--0153 and by the grant VEGA 1/0574/21.}
\fntext[mr]{Supported by Young Researchers Grant of Slovenian Research Agency.}
\fntext[arrs]{Supported by Slovenian Research Agency (Program P1--0383 and Projects J1--1692 and J1--3002).}

\address[im]{P. J. Šaf\'{a}rik University, Faculty of Science, Košice, Slovakia}
\address[ib]{Faculty of Information Studies, Novo mesto, Slovenia}
\address[if]{University of Ljubljana, Faculty of Mathematics and Physics, Ljubljana, Slovenia}

\begin{abstract}
	An $\ell$-facial edge-coloring of a plane graph is a coloring of its edges such that 
	any two edges at distance at most $\ell$ on a boundary walk of any face receive distinct colors.
	It is the edge-coloring variant of the $\ell$-facial vertex coloring, 
	which arose as a generalization of the well-known cyclic coloring.	
	It is conjectured that at most $3\ell + 1$ colors suffice for an $\ell$-facial edge-coloring of any plane graph.
	The conjecture has only been confirmed for $\ell \le 2$, and in this paper, 
	we prove its validity for $\ell = 3$.
\end{abstract}

\begin{keyword}
	$\ell$-facial edge-coloring \sep facial coloring \sep cyclic coloring \sep Facial Coloring Conjecture
\end{keyword}

\end{frontmatter}

\section{Introduction}

An \textit{$\ell$-facial edge-coloring} ($\ell$-FEC) of a plane graph $G$ (with loops and parallel edges allowed) 
is a not necessarily proper edge-coloring of its edges 
such that all the edges on a facial trail of length at most $\ell + 1$ receive distinct colors.
The minimum number of colors for which $G$ admits an $\ell$-facial edge-coloring is called the \textit{$\ell$-facial chromatic index of $G$}, 
denoted by $\chife(G)$.

This type of coloring was introduced in~\cite{LuzMocSotSkrSug15}, as the edge-coloring variant 
of the $\ell$-facial vertex coloring~\cite{KraMadSkr05}, 
which is a generalization of the cyclic coloring~\cite{OrePlu69}; 
the latter being a vertex coloring of a plane graph in which all the vertices incident with the same face receive distinct colors.
The cyclic coloring and the $\ell$-facial vertex coloring received a lot of attention,
but there are still many open problems regarding these two topics (see~\cite{CzaHorJen21} for a comprehensive survey 
and~\cite{DvoHebHlaKraNoe21} for the most recent results).
Particularly, Kr\'{a}l' et al.~\cite{KraMadSkr05} conjectured that at most $3\ell+1$ colors are required 
for an $\ell$-facial vertex coloring of any plane graph with the bound being tight by the
plane embeddings of $K_4$, in which three edges adjacent to the same vertex are subdivided $\ell-1$ times.
All the cases for $\ell \ge 2$ are still open, whereas the case $\ell = 1$ is implied by the Four Color Theorem.

Note that an $\ell$-facial edge-coloring of a plane graph $G$ corresponds to an $\ell$-facial vertex coloring 
of the medial graph $M(G)$ of $G$; i.e., the graph with the vertex set $V(M(G)) = E(G)$ 
and two vertices $u$ and $v$ of $M(G)$ being connected with $k$ edges if $u$ and $v$ 
correspond to two adjacent edges of $G$ incident with $k$ common faces in $G$.
Yet, the upper bound for $\ell$-facial chromatic index seems to be the same as for the vertex variant.
Namely, in~\cite{LuzMocSotSkrSug15}, the authors proposed the following conjecture.
\begin{conjecture}[Facial Edge-Coloring Conjecture]
	\label{con:facial edge-coloring}
	Every plane graph admits an $\ell$-facial edge-coloring with at most $3\ell + 1$ colors for every $\ell\ge 1$.
\end{conjecture}
\noindent If true, the conjectured bound is tight and achieved by the graphs depicted in Figure~\ref{fig:3l+1graph}.
\begin{figure}[htpb!]
	$$
		\includegraphics{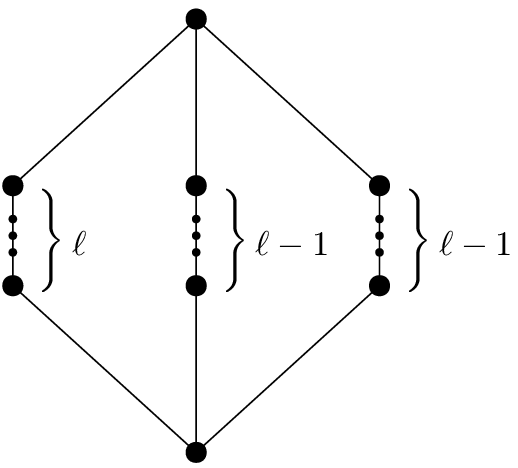}
	$$
	\caption{A graph $G$ with $\chife(G) = 3\ell + 1$.}
	\label{fig:3l+1graph}
\end{figure}

As mentioned above, for the case $\ell = 1$, Conjecture~\ref{con:facial edge-coloring} is implied by the Four Color Theorem,
and for the case $\ell=2$, it has been confirmed in~\cite{LuzMocSotSkrSug15}.
In this paper, we prove that Conjecture~\ref{con:facial edge-coloring} holds also for the case $\ell = 3$.
\begin{theorem}
	\label{thm:main}
	Every plane graph admits a $3$-facial edge-coloring with at most $10$ colors.
\end{theorem}
Note that the theorem holds for graphs with loops and parallel edges (the so-called \textit{pseudographs}).

The paper is structured as follows. 
In Section~\ref{sec:prel}, we introduce notation and present auxiliary results used for proving Theorem~\ref{thm:main}.
The proof of the theorem is given in Section~\ref{sec:proof}, and
we conclude the paper with a discussion on limitations of our approach and directions for further work.

\section{Preliminaries}
\label{sec:prel}

In this section, we define notions and present auxiliary results
that we are using in our proof.
In figures, by full circles we depict the vertices with a given degree, 
while empty circles denote vertices with arbitrary degrees.
We denote the set of consecutive integers from $p$ to $q$ by $[p,q]$, 
i.e., $[p,q] = \set{z \in \mathbb{Z} \ : \ p \le z \le q}$.

We denote the degree of a vertex $v$ by $d(v)$ and the length of a face $\alpha$ by $\ell(\alpha)$.
A vertex of degree $k$ (at least $k$, at most $k$) is called a {\em $k$-vertex} (a {\em $k^+$-vertex}, a {\em $k^-$-vertex}, respectively).
Similarly, a face of length $k$ (at least $k$, at most $k$) is called a {\em $k$-face} (a {\em $k^+$-face}, a {\em $k^-$-face}, respectively).
Note that the length of a face $\alpha$ in a $2$-connected plane graph is the number of edges 
(as well as the number of vertices) incident with $\alpha$.

Two edges are \textit{at facial-distance $k$} if they are at distance $k$ on some facial trail and $k$ is minimum satisfying the property;
two edges are \textit{$k$-facially adjacent} or \textit{within facial-distance $k$} if they are at facial-distance at most $k$. 
Note that we define the distance between two edges in $G$ as the distance between the corresponding vertices in the line graph of $G$;
in particular, the distance between adjacent edges equals $1$.
The \textit{$k$-facial neighborhood} of an edge $e$ is the set of all edges that are $k$-facially adjacent to $e$.

We define the \textit{distance between a vertex $v$ and an edge $e$} as the minimum distance from $v$ to any endvertex of $e$.

A $k$-vertex (a $k^+$-vertex) adjacent to a vertex $v$ is a {\em $k$-neighbor} (a {\em $k^+$-neighbor}) of $v$.
A {\em $k$-thread} is a subgraph in $G$, isomorphic to the path $P_k$, in which all vertices have degree $2$ in $G$.
When considering a $2$-thread composed of vertices $u$ and $v$, we denote it the \textit{$2$-thread $(u,v)$}.
A $k$-thread is {\em incident with a face $\alpha$} if all its vertices are incident with $\alpha$.
A $2$-thread $(u,v)$ is $\ell$-facially adjacent to a vertex $w$ if there is a facial trail of length at most $\ell$ 
between $u$ and $w$ or between $v$ and $w$.

The number of $2$-vertices adjacent to a vertex $v$ (incident with a face $\alpha$) is denoted $n_2(v)$ ($n_2(\alpha)$, respectively).

By \textit{contracting a face $\alpha$} of a graph $G$ we mean contracting step by step all the edges on the boundary of $\alpha$, 
i.e., removing the edges and identifying the vertices of $\alpha$.
We denote the obtained graph by $G/\alpha$.

Let $\sigma$ be a partial 3-FEC of a graph $G$ with the color set $C$. A color $c \in C$ is {\em $\sigma$-available}
(or {\em available} if $\sigma$ is evident from the context) for a non-colored edge $e \in E(G)$ provided that the set of colors 
of the edges $3$-facially adjacent to $e$ does not contain $c$.
The set of $\sigma$-available colors for $e$ is denoted $A_\sigma(e)$ or $A(e)$ for short.
Given a set $E \subseteq E(G)$ of non-colored edges, the set
$$
	A_\sigma(E) = \bigcup_{e \in E} A_\sigma(e)
$$
is called the set of $\sigma$-available colors for $E$.
The just introduced notion is used mainly for $E = E(\alpha)$ if all edges in $E(\alpha)$
(i.e., edges incident with a face $\alpha$ of $G$) are non-colored, and the notation $A_\sigma(E(\alpha))$ is then simplified to $A(\alpha)$.

A map $L$ is a \textit{list-assignment} for a graph $G$ 
if it assigns a list $L(v)$ of colors to each vertex $v$ of $G$. 
If $G$ admits a proper vertex coloring $\sigma_L$ such that $\sigma_L(v) \in L(v)$ for all vertices in $V(G)$, 
then $G$ is \textit{$L$-colorable} and $\sigma_L$ is an \textit{$L$-coloring} of $G$. 
The graph $G$ is \textit{$k$-choosable} if it is $L$-colorable for every list-assignment $L$, 
where $|L(v)| \ge k$, for every $v \in V(G)$. 

We make use of the following generalization of Brooks' theorem to list coloring.
\begin{theorem}[Borodin~\cite{Bor77}; Erd\H{o}s, Rubin, Taylor~\cite{ErdRubTay80}]
	\label{thm:list-assignment}
	Let $G$ be a connected graph. 
	Suppose that $L$ is a list-assignment where $|L(v)|\ge d(v)$ for each $v\in V(G)$. If
	\begin{itemize}
		\item{} $|L(v)| > d(v)$ for some vertex $v$, or
		\item{} $G$ contains a block which is neither a complete graph nor an induced odd cycle (i.e.,~$G$ is not a Gallai tree),
	\end{itemize}
	then $G$ admits an $L$-coloring.
\end{theorem}
In our proofs, for a configuration of edges being colored we create a conflict graph, in which every edge is represented by a vertex,
and two vertices are adjacent
if the corresponding edges are within facial-distance $3$. The map $L$ assigns to every edge its list of available colors.

A useful tool in proving coloring results is also Hall's Theorem,
which guarantees distinct colors for a set of vertices.
\begin{theorem}[Hall~\cite{Hal35}]
	\label{thm:Hall}
	A bipartite graph with partition sets $A$ and $B$ admits a matching that covers every vertex
	of $A$ if and only if for every set $S \subseteq A$ the number of vertices of $B$ with a neighbor in $S$ is at least $|S|$.
\end{theorem}
In other words, if a partial proper vertex coloring $\sigma$ of a graph $G$ 
with a color set $C$ and $n$ non-colored vertices is such that for every $k \in [1,n]$ 
and for every set $S \subseteq V(G)$ of $k$ non-colored vertices 
there is a set $A(S) \subseteq C$ with $|A(S)| \ge k$ 
such that each color $c \in A(S)$ can be used as a color for at least one vertex $v \in S$ 
(i.e., colors of the colored neighbors of $v$ are distinct from $c$), then $\sigma$ 
can be extended to a proper vertex coloring of the whole $G$.
Let us note that Theorem~\ref{thm:Hall} was already successfully used in the context of facial colorings~\cite{HavKraSerSkr10}.

Finally, we recall another tool for determining if one can always choose colors from the lists of 
available colors such that all conflicts are avoided. 
The result, due to Alon~\cite{Alo99}, is also referred to as the \textit{Combinatorial Nullstellensatz}.
\begin{theorem}[Alon~\cite{Alo99}]
	\label{thm:nullstellensatz}
	Let $\mathbb{F}$ be an arbitrary field, and let $P = P(X_1,\ldots ,X_n)$ be a polynomial in $\mathbb{F}[X_1,\ldots ,X_n]$.
	Suppose that the coefficient of a monomial $\prod_{i=1}^n X_i^{k_i}$, 
	where each $k_i$ is a nonnegative integer, is non-zero in $P$ and the degree $\deg(P)$ of $P$ equals $\sum_{i=1}^n k_i$.
	Moreover, if $S_1,\ldots ,S_n$ are any subsets of $\mathbb{F}$ with $|S_i| > k_i$, 
	then there exist $s_1\in S_1,\ldots ,s_n\in S_n$ such that $P(s_1,\ldots ,s_n)\not = 0$.
\end{theorem}
Namely, in our particular case of the $3$-facial edge-coloring, 
we assign a variable $X_i$ to every edge $e_i$ that we want to color (for $1 \le i \le k$),
and define a polynomial $P(X_1,\dots,X_k)$ such that every pair of $3$-facially adjacent edges is represented by a term $(X_i - X_j)$ in $P$.
If there is a monomial (of a proper degree) of $P$ with a non-zero coefficient, 
then there exists a coloring of the considered edges.

\section{Proof of \cref{thm:main}}
\label{sec:proof}

We prove the theorem by contradiction; 
namely, we suppose that there exists a minimal counterexample to the theorem, 
i.e., a graph $G$ (minimal according to the number of vertices)
that does not admit a \fec with at most $10$ colors.
We first determine some structural properties of $G$ in Section~\ref{sec:struct}, 
and then, in Section~\ref{sec:disc}, we use the discharging method to show that a graph $G$ with specified properties cannot exist, 
hence obtaining a contradiction.

\subsection{Structure of a minimal counterexample}
\label{sec:struct}

First, we show that there are no cut vertices in $G$.
Note that throughout the paper, for simplicity, we do not distinguish vertices and edges
of $G$ and the graphs obtained by modifying $G$.
\begin{lemma}
	\label{lem:2-con}
	$G$ is $2$-connected.
\end{lemma}

\begin{proof}
	Suppose the contrary and let $v$ be a cut vertex of $G$.  	
	There exists a component $H$ of $G - v$ such that the vertex $v$ is in the subgraph $G_1$ of $G$ 
	induced by the vertex set $V(H) \cup \set{v}$ incident with the unbounded face. 
	Let $G_2$ be the subgraph of $G$ induced by the vertex set $V(G) \setminus V(H)$. 
	By the minimality of G, there exist a \fec $\sigma_1$ of $G_1$ 
	and a \fec $\sigma_2$ of $G_2$ with the same set $C$ of at most $10$ colors.
	
	Consider the set $E_1^1$ of edges of the unbounded face of $G_1$ that are incident
	with $v$ (note that $1 \le |E_1^1| \le 2$) and the set $E_1^j$ of edges of $G_1$ 
	that are in $G_1$ at facial-distance $j-1$ from the closest edge of $E_1^1$, for $j = 2,3$. 
	Furthermore, consider the set $E_2^j$ of edges of $G_2$ that are in $G$ at facial-distance $j$ 
	from the closest edge of $E_1^1$, for $j = 1,2,3$. 
	For the set $C_i^j$ of colors of edges in $E_i^j$ we have $|C_i^j| \le 2$, 
	and we may assume without loss of generality that $|C_1'| \le |C_2'|$ 
	for $C_i' = C_i^1 \cup C_i^2 \cup C_i^3$, for $i = 1,2$.
	If $|C_1'| + |C_2'| \le 10$, then (again without loss of generality)
	$C_1' \cap C_2' = \emptyset$ and so the common extension of $\sigma_1$ and $\sigma_2$ 
	is a \fec of G with the set of colors $C$, a contradiction.
	
	So, $5 \le |C_1'| \le 6$, $|C_2'|=6$, $|C_2^1| = |C_2^2| = |C_2^3| = 2$ and $E_2^3 = \set{e_1,e_2}$.
	Let $C(e)$/$C[c]$ for $e \in E(G_2)$/$c \in C$ be the color class of $\sigma_2$ containing the edges of $G_2$ colored with $\sigma_2(e)$/$c$.
	Since $1 \le p = |\set{\sigma_2(e_1)} \cup \set{\sigma_2(e_2)}| \le 2$, 
	and the color set $C^* = C \setminus (C_1^1 \cup C_1^2 \cup C_1^3 \cup C_2^1 \cup C_2^2)$ is of size at least $10 - 4 \cdot 2 \ge p$,
	there is a $p$-element set $\set{c_j \ : \ j \in [1,p]} \subseteq C^*$.
	Now recolor $G_2$ using the permutation $\pi$ of $C$ that induces the permutation of color classes of $\sigma_2$,
	under which the color classes $C(\sigma_2(e_j))$ and $C[c_j]$ are interchanged for each $j \in [1,p]$,
	and all remaining color classes are fixed.
	It is easy to see that the common extension of $\sigma_1$ and $\pi \circ \sigma_2$ is a \fec of $G$ with the color set $C$,
	a contradiction.
\end{proof}

From the above, we also infer that there are no pendant vertices in $G$.
\begin{corollary}
	\label{lem:min-deg-2}
	The minimum degree of $G$ is at least $2$.
\end{corollary}

Similarly, using \cref{lem:2-con}, we show there are no loops in $G$.
\begin{lemma}
	\label{lem:loops}
	$G$ is loopless.
\end{lemma}

\begin{proof}
	Suppose, to the contrary, that there is a loop $e$ in $G$.
	If $e$ bounds a $1$-face, then it is $3$-facially adjacent to at most $6$ edges in $G$
	and thus we obtain a \fec with at most $10$ colors of $G$ by removing $e$,
	coloring the obtained graph, and finally coloring $e$ with one of at least $4$ available colors.
	On the other hand, if $e$ does not bound a $1$-face, then its unique endvertex is a cut vertex in $G$,
	a contradiction to \cref{lem:2-con}.
\end{proof}

Hence, in $G$, every $k$-face is incident with $k$ distinct vertices and with $k$ distinct edges.
In the rest of the paper, we will mainly deal with $2$-vertices and small faces in $G$.
\begin{lemma}
	\label{lem:4vert2verts}
	A $4$-vertex in $G$ has at most three $2$-neighbors.
\end{lemma}

\begin{proof}
	Suppose the contrary and let $v$ be a $4$-vertex adjacent to four $2$-vertices $v_1, v_2, v_3$ and $v_4$ in a clockwise order.
	Let $v_{i+4}$ be the other neighbor of $v_i$, for $1\le i\le 4$.
	Let $G'$ be the graph obtained from $G$ by deleting the vertices $v,v_1,v_2,v_3$ and $v_4$.
	By the minimality of $G$, there exists a \fec coloring $\sigma$ of $G'$ with at most $10$ colors.
	Notice that each of the edges $vv_i$ and $v_iv_{i+4}$ has at least $4$ available colors.
	Let $X_j$, $1\le j\le 8$ be a variable associated with the edge $vv_j$ if $j\le 4$ and the edge $v_{j-4}v_j$ otherwise.
	Let us now define the following polynomial, simulating the conflicts between the non-colored edges:
	\begin{align*}
		F(X_1,\ldots ,X_8) = & (X_1 - X_2)(X_1 - X_4)(X_1 - X_5)(X_1 - X_6)(X_1 - X_8)\\
						   \cdot & (X_2 - X_3)(X_2 - X_5)(X_2 - X_6)(X_2 - X_7)(X_3 - X_4)\\
						   \cdot & (X_3 - X_6)(X_3 - X_7)(X_3 - X_8)(X_4 - X_5)(X_4 - X_7)\\
						   \cdot & (X_4 - X_8)(X_5 - X_6)(X_5 - X_8)(X_6 - X_7)(X_7 - X_8).
	\end{align*}
	The coefficient of the monomial $X_1^3X_2^3X_3^3X_4^3X_5^2X_6^2X_7^2X_8^2$ in $F(X_1,\ldots ,X_8)$ is equal to $6$\footnote{We verified the values 
	of the coefficients with a computer program.}, 
	and thus by Theorem~\ref{thm:nullstellensatz} we can extend the coloring $\sigma$ to the coloring of $G$ using at most $10$ colors.
\end{proof}

Let $C$ be a cycle in $G$. 
We denote by $\int(C)$ the graph induced by the vertices lying strictly in the interior of $C$.
Similarly, we denote by $\ext(C)$ the graph induced by the vertices lying strictly in the exterior of $C$.
We say that $C$ is a \textit{separating cycle} if both, $\int(C)$ and $\ext(C)$, contain at least one vertex.
\begin{lemma}
	\label{lem:separating-cycle}
	There is no separating cycle of length at most $7$ in $G$.
\end{lemma}

\begin{proof}
	Suppose the contrary and let $C$ be a separating cycle of length at most $7$. 
	Let $G_1$ be the subgraph of $G$ induced by the vertex set $V(\int(C))\cup V(C)$ 
	and let $G_2$ be the subgraph of $G$ induced by the vertex set $V(\ext(C))\cup V(C)$. 
	By the minimality of $G$, there exists a \fec $\sigma_1$ and a \fec $\sigma_2$ of $G_1$ and $G_2$, respectively, using the same set of at most $10$ colors. 
	Notice that, since the length of $C$ is at most $7$, every edge of $C$ is $3$-facially adjacent to all the other edges of $C$ in both $G_1$ and $G_2$. 
	Thus, all the edges of $C$ receive distinct colors in both $\sigma_1$ and $\sigma_2$. 
	Hence, permuting the colors in $\sigma_1$ such that the colors of the edges of $C$ coincide in $\sigma_1$ and in $\sigma_2$, 
	results in a \fec of $G$ with at most $10$ colors.
\end{proof}

Next, we show that $G$ does not contain small faces nor faces of length $8$.
\begin{lemma}
	\label{lem:face-length}
	Every face in $G$ is of length at least $5$.
\end{lemma}

\begin{proof}
	Suppose the contrary and let $\alpha$ be a face of $G$ of length at most $4$. 
	Let $G' = G/\alpha$ and let, by the minimality of $G$, $\sigma$ be a \fec of $G'$ using at most $10$ colors. 
	Next, observe that each edge of $\alpha$ is $3$-facially adjacent to at most six edges of $G'$ in $G$.
	Thus, each edge of $\alpha$ has at least $4$ available colors. 
	By \cref{thm:list-assignment}, we can therefore extend the coloring $\sigma$ to obtain a \fec of $G$ using at most $10$ colors.
\end{proof}

Note that Lemmas~\ref{lem:2-con}, \ref{lem:loops}, \ref{lem:separating-cycle}, and~\ref{lem:face-length} 
imply that $G$ is a simple graph.

\begin{lemma}
	\label{lem:8face}
	There are no $8$-faces in $G$.
\end{lemma}

In the proof of \cref{lem:8face} and several other proofs, we identify two edges of the same face $\alpha$ which are not in conflict in $G$.
The identification is always made in such a way that the resulting graph is still planar, 
i.e., the subgraph of $G$ induced by the edges of the involved face $\alpha$ is transformed to a dumbbell subgraph of the resulting graph.

\begin{proof}
	Suppose the contrary and let $\alpha$ be an $8$-face in $G$ and $e$ and $f$ be two edges at facial-distance $4$ on $\alpha$. 
	Let $G'$ be the graph obtained from $G$ by identifying the edges $e$ and $f$ and let $\sigma$ be a \fec of $G'$ using at most $10$ colors. 
	Observe that the edges $e$ and $f$ are not $3$-facially adjacent in $G$, 
	otherwise $G$ would contain 
	either a separating cycle of length at most $5$ (contradicting \cref{lem:separating-cycle}) or a $3$-face (contradicting \cref{lem:face-length}). 
	Therefore, after we uncolor every edge of $\alpha$ distinct from $e$ and $f$, 
	$\sigma$ induces a partial \fec of $G$ in which the edges $e$ and $f$ receive the same color. 
	
	To extend the coloring $\sigma$ to a coloring of $G$, notice that all six non-colored edges of $\alpha$ have at least $3$ available colors. 
	Furthermore, among those edges there are exactly three distinct pairs of edges at facial-distance $4$. 
	If we can color any such pair with the same color, then the remaining four edges will each have at least $2$ available colors.
	Furthermore, each of them is at facial-distance at most $3$ from exactly two other non-colored edges. 
	Applying \cref{thm:list-assignment}, we obtain a \fec using at most $10$ colors. 
	Therefore, we may assume that the union of available colors of any such pair is of size at least $6$, 
	with each edge having at least 3 available colors. 
	Thus, we can extend the coloring $\sigma$ to a \fec of $G$ by \cref{thm:Hall}.
\end{proof}

In the following lemmas, we give several properties of $2$-vertices in $G$.
First, we show that there are no (naturally defined) $3^+$-threads in $G$.
\begin{lemma}
	\label{lem:3thread}
	Every $2$-vertex in $G$ has at least one $3^+$-neighbor.
\end{lemma}

\begin{proof}
	Suppose to the contrary that $v$ is a $2$-vertex with neighbors $u_1$ and $u_2$,
	both being $2$-vertices.
	Let $G' = G/u_1v$ and let, by the minimality of $G$, $\sigma$ be a \fec of $G'$ using at most $10$ colors.
	Notice that facial-distances between the edges in $G$ are at least the distances between them in $G'$, and thus
	the coloring $\sigma$ induces a partial \fec of $G$ in which only the edge $u_1v$ is non-colored.
	However, there are only nine edges in the $3$-facial-neighborhood of $u_1v$, 
	and therefore at least one color is available for $u_1v$ 
	(to extend $\sigma$ to a \fec of $G$),
	a contradiction.
\end{proof}

\begin{lemma}
	\label{lem:2thread2vert}
	Let $(u,v)$ be a $2$-thread in $G$ incident with an $8^+$-face $\alpha$.
	Then, within facial-distance $3$ on the face $\alpha$, except from $u$, $v$ is adjacent only to $3^+$-vertices.	
\end{lemma}

\begin{proof}
	Suppose the contrary and let a $2$-thread $(u,v)$ be $3$-facially adjacent to a $2$-vertex $w \in \set{v_2,v_3}$ of $\alpha$.
	We use the labeling of vertices as depicted in Figure~\ref{fig:2thread2vert}.
	\begin{figure}[htp!]
		$$
			\includegraphics[scale=\figscale]{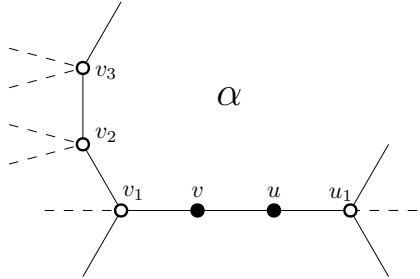}
		$$
		\caption{A reducible configuration with a $2$-thread and a $2$-vertex $w \in \set{v_2,v_3}$.}
		\label{fig:2thread2vert}
	\end{figure}	
	
	Let $G' = G/\set{uu_1,uv,vv_1,v_1v_2,v_2v_3}$ and let $\sigma$ be a \fec of $G'$.
	In the coloring of $G$ induced by $\sigma$, regardless which is $w$,
	we have $|A(v_2v_3)| \ge 2$, $|A(v_1v_2)| \ge 2$, $|A(vv_1)| \ge 4$, $|A(uv)| \ge 4$, and $|A(uu_1)| \ge 3$.
	If $A(uu_1) \cap A(v_2v_3) \neq \emptyset$, then we color $uu_1$ and $v_2v_3$ with the same color 
	(recall that they are not $3$-facially adjacent since $\alpha$ is an $8^+$-face),
	and color the remaining three edges by \cref{thm:Hall}.
	
	On the other hand, if $A(uu_1) \cap A(v_2v_3) = \emptyset$, 
	then in the union of available colors of the five non-colored edges we have at least $5$ colors, 
	and it is easy to see that again \cref{thm:Hall} can be applied to color all the edges of $G$, a contradiction.	
\end{proof}

Clearly, if every $2$-vertex incident with a $k$-face $\alpha$ has two $3^+$-neighbors, 
then $\alpha$ is incident with at most $\lfloor k/2 \rfloor$ vertices of degree $2$.
In the case of incident $2$-threads, we can further limit the number of $2$-vertices incident with a face using \cref{lem:2thread2vert}.
We will formalize this fact in the next corollary after defining some additional notions.

Let $n_2^t(\alpha)$ be the number of $2$-vertices incident with a face $\alpha$ that belong to $2$-threads.
A {\em $k$-path} of a face $\alpha$, $k \in \set{2,3^+}$, is a maximal (i.e., non-extendable) 
facial path in $\alpha$ composed of $k$-vertices. 
If $n_2(\alpha) > 0$, then the set $V(\alpha)$ of the vertices incident with $\alpha$ 
has, for some positive integer $p$, 
a partition $\set{V^i \ : \ i = 1,\dots,2p}$ 
such that the set $V^{2i-1}$ induces a $2$-path $P^{2i-1}$ of $\alpha$, 
and the set $V^{2i}$ induces a $3^+$-path $P^{2i}$ of $\alpha$ that follows $P^{2i-1}$
in the (say) clockwise orientation of $\alpha$ for each $i = 1,\dots,p$ 
(and $P^1$ follows $P^{2p}$). 
A {\em section} of $\alpha$ is a pair $(V^{2i-1},V^{2i})$, $i = 1,\dots,p$; 
the pair $(V^{2i-1}, V^{2i})$ is a {\em $j$-section} of $\alpha$ 
if $|V^{2i-1}| = j \in \set{1,2}$ (see \cref{lem:3thread}). 
Let $S_j(\alpha)$ denote the set of $j$-sections of $\alpha$, $j = 1,2$.
\begin{corollary}
	\label{cor:num2verts}
	For a $k$-face $\alpha$ of $G$, where $k \ge 8$ and $n_2(\alpha) > 0$, 
	we have 
	$$
		n_2(\alpha) \le \bigg\lfloor \frac{k}{2} \bigg\rfloor \quad\quad \textrm{and} \quad\quad |S_2(\alpha)| \le \bigg\lfloor \frac{k-2\cdot |S_1(\alpha)|}{5} \bigg\rfloor\,.
	$$
	Moreover, if $k = 11$ and $|S_2(\alpha)| > 0$, then $n_2(\alpha) \le 4$.
\end{corollary}

\begin{proof}
	Let $\set{V^i \ : \ i = 1,\dots,2p}$ be the partition of $V(\alpha)$ as defined in the above paragraph.
	If $(V^{2i-1}, V^{2i}) \in S_1(\alpha)$, then $|V^{2i}| \ge 1$.
	On the other hand, if $(V^{2i-1}, V^{2i}) \in S_2(\alpha)$, then, by \cref{lem:2thread2vert}, $|V^{2i}| \ge 3$.
	Therefore,
	\begin{align}
		\label{eq:n2}
		k 	= \sum_{i = 1}^p \big( |V^{2i-1}| + |V^{2i}| \big) 
			\ge 2|S_1(\alpha)| + 5|S_2(\alpha)| 				
			= 2\big(|S_1(\alpha)| + 2|S_2(\alpha)|\big) + |S_2(\alpha)|\,.
	\end{align}
	From~\eqref{eq:n2} we infer that
	$$
		n_2(\alpha) = |S_1(\alpha)| + 2|S_2(\alpha)| \le \frac{1}{2}\big(k - |S_2(\alpha)|\big) \le \frac{k}{2}\,,
	$$
	implying $n_2(\alpha) \le \big\lfloor \frac{k}{2} \big\rfloor$,
	as well as 
	$$
		|S_2(\alpha)| \le \frac{1}{5}\big( k - 2|S_1(\alpha)| \big)\,,
	$$
	implying $|S_2(\alpha)| \le \big\lfloor \frac{k-2\cdot |S_1(\alpha)|}{5} \big\rfloor$.
	
	Now suppose that $k = 11$ and $|S_2(\alpha)| > 0$, which implies $|S_2(\alpha)| = q \in \set{1,2}$,
	since otherwise $n_2(\alpha) \ge 2q \ge 6$ contradicting that $n_2(\alpha) \le \lfloor \frac{11}{2} \rfloor = 5$.
	
	If $q = 1$, then the number of $3^+$-vertices of $\alpha$, that are in $\alpha$ within facial-distance $3$ 
	from a vertex of the $2$-thread of $\alpha$, is 6 (by \cref{lem:2thread2vert}). 
	At most two of the three remaining vertices of $\alpha$ are $2$-vertices (by \cref{lem:3thread}),
	hence $|S_1(\alpha)| \le 2$ and $n_2(\alpha) \le 2 + 2\cdot 1 = 4$.
	
	If $q = 2$, consider an arbitrary vertex $x$ of $\alpha$ 
	that is not a part of a $2$-thread of $\alpha$. 
	The vertex $x$ is in $\alpha$ within facial-distance $2$ 
	from a vertex of at least one of the two $2$-threads of $\alpha$, 
	and so, 
	by \cref{lem:2thread2vert}, $x$ is a $3^+$-vertex. 
	This reasoning leads to $n_2(\alpha) = 2 \cdot 2 = 4$.
\end{proof}


In the next lemma, we show that presence of $3$-vertices in some cases enables recoloring of certain edges.
\begin{lemma}
	\label{lem:3vert-recolor}
	Let $uv$ be an edge with $d(u)=3$, and let $uu_1$, $uu_2$ be the other two edges incident with $u$.
	Consider a partial \fec of $G$, in which the edge $uv$ is, and the edges $uu_1$, $uu_2$ are not colored.
	If $|A(uu_1) \cap A(uu_2)| = k$ for some $k \ge 3$,
	then there are at least $k-2$ colors in $A(uu_1) \cap A(uu_2)$ 
	such that each can be used to recolor the edge $uv$ in such a way that the result is again a partial \fec of $G$.
\end{lemma}

\begin{proof}
	In the $3$-facial-neighborhood of $uv$, there are at most two edges which are not $3$-facially adjacent to $uu_1$ or $uu_2$, 
	which means that there are at least $k-2$ available colors for $uv$ from the intersection $A(uu_1) \cap A(uu_2)$.
\end{proof}

We continue by establishing properties about incidences of small faces.
First, we show that $5$-faces are not incident with small vertices.
\begin{lemma}
	\label{lem:5face4vertex}
	Every $5$-face in $G$ is incident only with $4^+$-vertices.
\end{lemma}

\begin{proof}
	Suppose the contrary and let $\alpha$ be a $5$-face of $G$ incident with a $3$-vertex $v_1$,	
	where the vertices are labeled as in Figure~\ref{fig:5face4verts}.
	(Note that $G$ is not the $5$-cycle $C_5$.)
	\begin{figure}[htp!]
		$$
			\includegraphics[scale=\figscale]{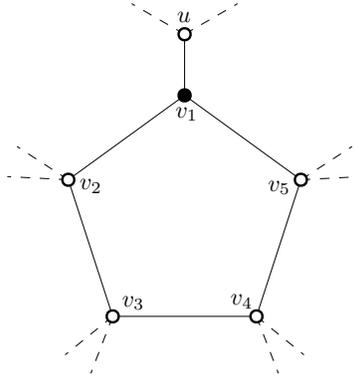}
		$$
		\caption{A reducible $5$-face incident with a $3$-vertex.}
		\label{fig:5face4verts}
	\end{figure}
	
	Let $\sigma$ be a \fec of $G'=G/\alpha$. 
	It induces a partial \fec of $G$ with the five edges of $\alpha$ being non-colored.
	Each of the non-colored edges has at least $4$ available colors.
	By \cref{thm:Hall}, if the union of the five sets of available colors contains at least $5$ distinct colors,
	then $\sigma$ can be extended to $G$.		
	Therefore, we may assume that $A(e)$ is the same for every $e \in E(\alpha)$, say $A(e) = [1,4]$. 
	In such a case, the face $\alpha$ is incident with $3^+$-vertices only: 
	if $d(v_i) = 2$ for some $i \in [2,5]$, 
	then both edges incident with $v_i$ have at least $5$ available colors. 
	Now, we recolor the edge $uv_1$ with a color $j \in A(v_1v_2) \cap A(v_1v_5) = [1,4]$. 
	(By \cref{lem:3vert-recolor}, there are at least $2$ possibilities for the choice of $j$.) 
	Recall that the $2$-connected graph $G$ contains neither $3$-faces nor separating cycles of length at most $5$. 
	Thus, the edges $v_2v_3$, $v_3v_4$, and $v_4v_5$ are not within facial-distance $3$ from the edge $uv_1$,
	and they retain $[1,4]$ as the set of available colors. 
	Furthermore, $\sigma(uv_1) \notin [1,4]$ replaces $j$ in the set of available colors 
	for the edges $v_1v_2$ and $v_1v_5$. 
	So, the coloring of $G'$ can be extended to $G$ using \cref{thm:Hall}, a contradiction.
\end{proof}

So, there are no $3^-$-vertices incident with $5$-faces.
On the other hand, there might be $2$-vertices incident with $6$-faces.
\begin{lemma}
	\label{lem:6face2thread}
	Both neighbors of a $2$-vertex incident with a $6$-face in $G$ are $4^+$-vertices.
\end{lemma}

\begin{proof}
	We divide the proof in two parts. First, we show that a $2$-vertex does not have a $2$-neighbor,
	i.e., there is no $2$-thread on a $6$-face.
	Suppose the contrary and let $\alpha$ be a $6$-face with an incident $2$-thread $(u,v)$ and let $G' = G/\alpha$.
	Then, $G'$ admits a \fec $\sigma$ with at most $10$ colors, which induces a \fec of $G$
	with only the edges of $\alpha$ being non-colored.
	The three edges incident with the vertices $u$ and $v$ have at least $6$ available colors, 
	and the other edges of $\alpha$ have at least $4$ available colors.
	It is easy to see that we can extend $\sigma$ to all edges of $G$ by applying \cref{thm:Hall}, a contradiction.
	
	Second, suppose that a $2$-vertex $v$ of a $6$-face $\alpha$ is adjacent to a $3$-vertex $u$.
	Let $u_1$ be the neighbor of $u$, distinct from $v$, which is incident with $\alpha$, 
	and $u_2$ the third neighbor of $u$.
	Again, consider $G' = G/\alpha$ and a \fec $\sigma$ of $G'$ using at most $10$ colors.
	In the coloring of $G$ induced by $\sigma$, only the edges of $\alpha$ are non-colored.
	Every non-colored edge has at least $4$ available colors, 
	while the two edges incident with $v$ have at least $5$ available colors.
	Hence, if the set $A(\alpha)$ contains at least $6$ colors, then we can apply \cref{thm:Hall} and we are done.
	
	Thus, we may assume that $A(\alpha)$ contains precisely $5$ colors, say $[1,5]$.
	Notice that the sets of available colors on the edges of $\alpha$ not incident with $v$ are not necessarily the same.
	However, since $|A(uv)| = 5$, the intersection of $A(e) \cap A(uv)$, for any $e \in E(\alpha)$, contains at least $4$ colors.	
	Therefore, from among at least $2$ colors that can be used to recolor the edge $uu_2$ by \cref{lem:3vert-recolor}, 
	at least one, say $j$, appears in $A(e)$ for some $e \in E(\alpha) \setminus \set{uu_1,uv}$. 
	Then, after recoloring the edge $uu_2$ with the color $j$, 
	the new set of available colors for $E(\alpha)$ is of size $6$, 
	and we can apply \cref{thm:Hall} to find a \fec of $G$ with at most $10$ colors, a contradiction.
\end{proof}

From Lemmas~\ref{lem:face-length}, \ref{lem:5face4vertex}, and~\ref{lem:6face2thread} we obtain the following Corollary.
\begin{corollary}
	\label{cor:2thread7face}
	No $2$-thread in $G$ is incident with a $6^-$-face.
\end{corollary}

\begin{lemma}
	\label{lem:6face2vert}
	A $2$-vertex in $G$ is incident with at least one $7^+$-face.
\end{lemma}

\begin{proof}
	We again proceed by contradiction.
	Since $2$-vertices are not incident with $5^-$-faces by Lemmas~\ref{lem:face-length} and~\ref{lem:5face4vertex},
	suppose that there is a $2$-vertex $v$ in $G$ incident with two $6$-faces.
	Let $G' = G-v$. 
	By the minimality, there is a \fec $\sigma$ of $G'$ using at most $10$ colors.	
	Consider now the coloring of $G$ induced by $\sigma$, in which only the two edges incident with $v$
	remain non-colored.
	Each of the two edges has at least $2$ available colors, so we can color them,
	and thus extend $\sigma$ to all edges of $G$, a contradiction.
\end{proof}

Now, we focus on $7$-faces incident with $2$-vertices.
We begin with $7$-faces incident with a $2$-thread.

\begin{lemma}
	\label{lem:7face2th3nei}
	Every $2$-thread incident with a $7$-face in $G$ has at least one $4^+$-neighbor.
\end{lemma}

\begin{proof}
	Suppose the contrary and let $\alpha$ be a $7$-face incident with a $2$-thread $(v_2,v_3)$,
	where the other neighbors of $v_2$ and $v_3$ ($v_1$ and $v_4$, respectively) are both $3$-vertices.
	We label the vertices as depicted in Figure~\ref{fig:7face2th3nei}.
	\begin{figure}[htp!]
		$$
			\includegraphics[scale=\figscale]{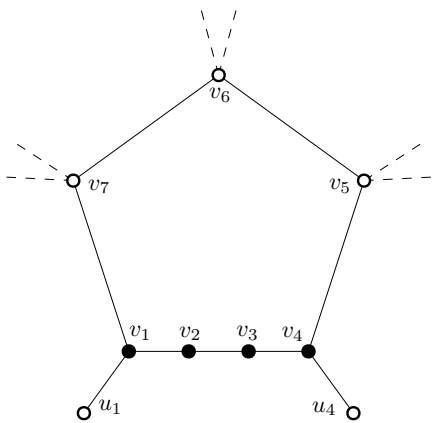}
		$$
		\caption{A reducible $7$-face incident with a $2$-thread with two $3$-neighbors.}
		\label{fig:7face2th3nei}
	\end{figure}
	
	Let $G' = G/\alpha$ and let $\sigma$ be a \fec of $G'$.
	In the partial coloring of $G$ induced by $\sigma$, only the edges of $\alpha$ are non-colored,
	and the number of available colors is at least $4$ for arbitrary non-colored edge, 
	while it is at least $6$ for the three edges incident with $v_2$ and/or $v_3$.
	It is easy to verify that if the set $A(\alpha)$ of available colors contains at least $7$ colors, 
	then we can complete the coloring by \cref{thm:Hall}.
	Thus we may assume that $|A(\alpha)| = 6$, say $A(\alpha) = [1,6]$.
	Additionally, we may assume that $\sigma(u_1v_1) = 7$.
	
	So, $|A(v_1v_2) \cap A(v_1v_7)| \ge 4$, and we can recolor $u_1v_1$ with a color 
	from $I = A(v_1v_2) \cap A(v_1v_7) \cap A(u_1v_1)$, since $|I| \ge 2$ by \cref{lem:3vert-recolor}.
	If there is an edge $e$ of $\alpha$ which is not $3$-facially adjacent to $u_1v_1$ 
	and $A(e)$ contains a color from $I$, say $1$, then we can recolor $u_1v_1$ with $1$.
	The new set of available colors for $E(\alpha)$ is then $[1,7]$, 
	and hence we can apply \cref{thm:Hall} to find a contradictory \fec of $G$.
	
	Note that if $|I| \ge 3$, then we can always find a suitable edge $e$.
	Therefore, we may assume $|I| = 2$, say $I = \set{1,2}$, 
	hence $d(v_7) \ge 3$, and, by symmetry, $d(v_5) \ge 3$. 
	Therefore, by Lemmas~\ref{lem:face-length} and~\ref{lem:3thread}, 
	there is no edge in the set $\set{v_4v_5,v_5v_6,v_6v_7}$ that is $3$-facially adjacent to $u_1v_1$; 
	thus, we have, say, $A(v_1v_7) = [1,4]$ and $A(v_4v_5) = A(v_5v_6) = A(v_6v_7) = [3,6]$.
	Analogously as above, at least two colors from $A(v_3v_4)\cap A(v_4v_5)$ 
	can be used to recolor $u_4v_4$ by \cref{lem:3vert-recolor}. 
	Since the color of $A(\alpha)$ involved in the recoloring is still available for $v_6v_7$, 
	the new set of available colors for $E(\alpha)$, 
	namely $[1,6] \cup \set{\sigma(u_4v_4)}$, is of size $7$, a contradiction. 
\end{proof}

\begin{lemma}
	\label{lem:7face2thread7fac}
	A $2$-thread in $G$ is incident with at most one $7$-face.
\end{lemma}

\begin{proof}
	Suppose the contrary and let $(v_1,v_2)$ be a $2$-thread incident with two $7$-faces $\alpha$ and $\alpha'$.
	Let $G' = G \setminus \set{v_1,v_2}$.
	By \cref{lem:8face}, there is a \fec $\sigma$ of $G'$ using at most $10$ colors such that two edges 
	of the face in $G'$ corresponding to the faces $\alpha$ and $\alpha'$ in $G$ have the same color assigned.
	This means that in the coloring of $G$ induced by $\sigma$, each of the three non-colored edges (the edges incident with the $2$-thread)
	have at least $3$ available colors, and therefore we can extend $\sigma$ to all edges of $G$, a contradiction.
\end{proof}

\begin{lemma}
	\label{lem:7face2thread2vert}
	Let $\alpha$ be a $7$-face in $G$ with a $2$-thread $(v_2,v_3)$ and at least one $2$-vertex $v$ distinct from $v_2$ and $v_3$.
	Then, every $2$-vertex incident with $\alpha$ has a $2$-neighbor and a $4^+$-neighbor or two $4^+$-neighbors.
\end{lemma}

\begin{proof}
	Suppose the contrary and let $\alpha$ be a $7$-face with the vertices labeled as in Figure~\ref{fig:7face2th2vert},
	with a $2$-vertex incident with a $3$-vertex.
	We present three possibilities (up to symmetry) for a neighboring $2$-vertex and a $3$-vertex;
	namely, in the case $(a)$, there is a $3$-neighbor of a $2$-thread, and in the cases $(b)$ and $(c)$ 
	a $3$-neighbor of a $2$-vertex $v$, which is not a part of the $2$-thread $(v_2,v_3)$.
	By \cref{lem:3thread}, we may assume that $v \in \set{v_5,v_6,v_7}$.
	\begin{figure}[htbp!]
		\centering
		\begin{subfigure}[b]{.32\textwidth}
			$$
				\includegraphics[scale=\figscale]{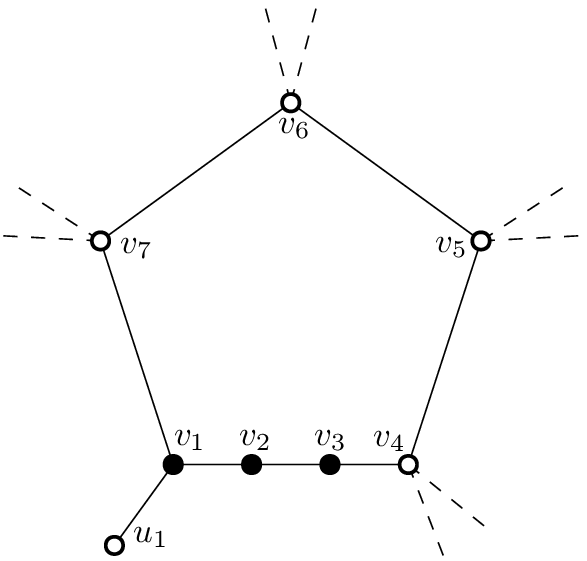}
			$$
			\caption{}
			\label{fig:7face2th2vertA}
		\end{subfigure}	
		\begin{subfigure}[b]{.33\textwidth}
			$$
				\includegraphics[scale=\figscale]{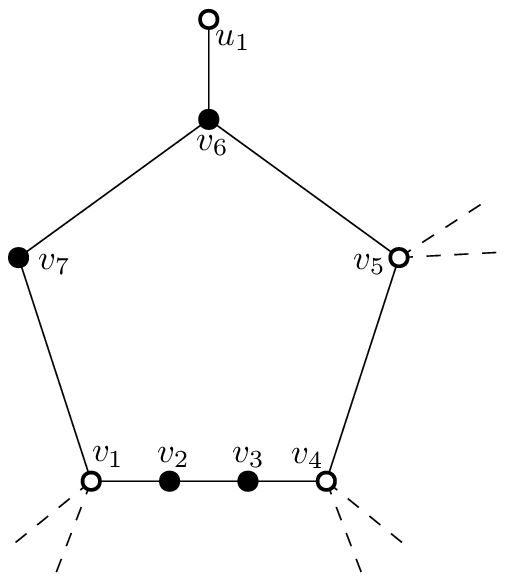}
			$$
			\caption{}
			\label{fig:7face2th2vertB}
		\end{subfigure}	
		\begin{subfigure}[b]{.33\textwidth}
			$$
				\includegraphics[scale=\figscale]{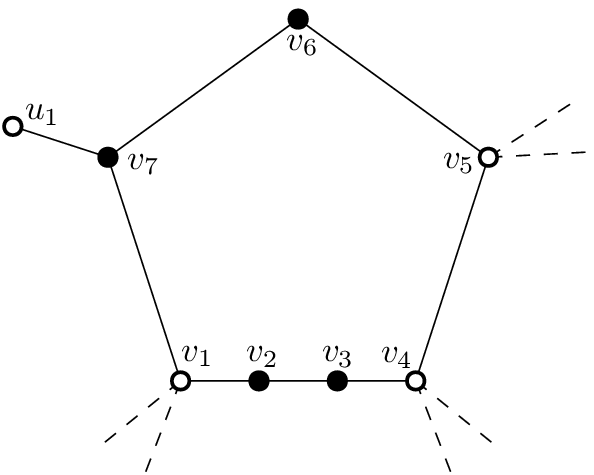}
			$$
			\caption{}
			\label{fig:7face2th2vertC}
		\end{subfigure}			
		\caption{The three possible configurations of a $7$-face incident with a $2$-thread, a $2$-vertex, and a $3$-vertex.}
		\label{fig:7face2th2vert}
	\end{figure}
	
	We prove the lemma for all three cases at once.
	Suppose to the contrary that $\alpha$ (one of the three possible ones) exists in $G$.
	Let $G' = G / \alpha$ and let $\sigma$ be a \fec of $G'$ with at most $10$ colors.
	In the coloring of $G$ induced by $\sigma$, only the edges of $\alpha$ are non-colored.
	Notice that the three edges incident with the $2$-thread $(v_2,v_3)$ have at least $6$ available colors,
	the two edges incident with $v$ have at least $5$, and the remaining two edges have at least $4$.
	From this it follows that if $|A(\alpha)| \ge 7$, then \cref{thm:Hall} applies, 
	and we can color all the edges of $\alpha$ with a different color, hence extending $\sigma$ to $G$, a contradiction.
	
	So, we may assume that $|A(\alpha)| = 6$.
	Denote by $v'$ the $3$-vertex adjacent to $u_1$ (hence, $v' \in \set{v_1,v_6,v_7}$), 
	and let $v_1',v_2'$ be the two neighbors of $v'$ on $\alpha$.
	We claim that there exists an edge $e'$ in $\alpha$, which is not $3$-facially adjacent to $u_1v'$,
	such that $|A(v'v_1') \cap A(v'v_2') \cap A(e')| \ge 3$.
	Note first that by the above argument on the number of available colors, 
	the intersection of available colors of any two edges of $\alpha$, 
	where at least one of them is incident with a $2$-vertex, is at least of size $3$.
	If $v'$ is not $v_1$, then $e' = v_1v_2$ is not $3$-facially adjacent to $u_1v'$ by \cref{lem:separating-cycle},
	and since $|A(v_1v_2)|=6$, the claim follows.
	Otherwise, if $v' = v_1$, we may assume $v_1' = v_2$, 
	and we choose $e' \in \set{v_4v_5,v_6v_7}$ in such a way that $e'$ is incident with a $2$-vertex.
	Similarly as above, since $|A(v_1v_2)|=6$ and $|A(v'v_2') \cap A(e')| \ge 3$, the claim follows.
	Now, by \cref{lem:3vert-recolor}, 
	recoloring $u_1v'$ with a color $c \in A(v'v_1') \cap A(v'v_2') \cap A(e')$ 
	introduces the color $\sigma(u_1v') \notin A(\alpha)$ to the set of available colors for $E(\alpha)$. 
	Since the color $c$ is still available for $e'$, 
	the new \fec of $G'$ can be extended to $G$ by \cref{thm:Hall}, a contradiction.
\end{proof}

Note that there might be a $7$-face in $G$ incident with two $2$-threads.

\begin{lemma}
	\label{lem:7face6face}
	If a $7$-face $\alpha$ in $G$ is incident with at least two $2$-vertices but no $2$-thread,
	then every $2$-vertex incident with $\alpha$ has at least one $4^+$-neighbor.
\end{lemma}

\begin{proof}
	Suppose the contrary and let $\alpha$ be a $7$-face in $G$ incident with at least two $2$-vertices,
	where one of them, call it $v_1$, has two $3$-neighbors. 
	\begin{figure}[htp!]
		\centering
		\begin{subfigure}[b]{.45\textwidth}
			$$
				\includegraphics[scale=\figscale]{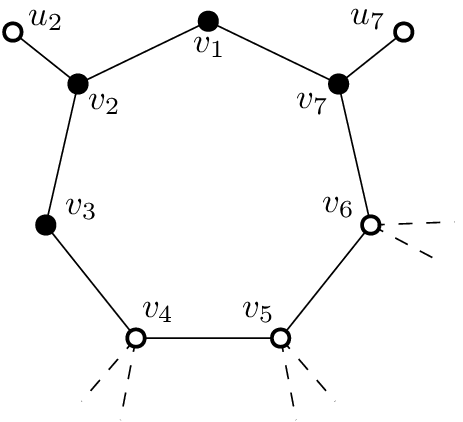}
			$$
			\caption{}
			\label{fig:7face2vertsA}
		\end{subfigure}	
		\begin{subfigure}[b]{.45\textwidth}
			$$
				\includegraphics[scale=\figscale]{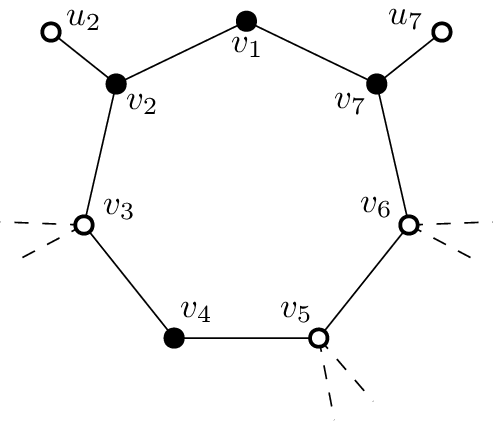}
			$$
			\caption{}
			\label{fig:7face2vertsB}
		\end{subfigure}		
		\caption{A $7$-face with at least two incident $2$-vertices, where one of them has two $3$-neighbors.}
		\label{fig:7face2verts}
	\end{figure}		
	Note that by symmetry we may also assume that either $v_3$ or $v_4$ is a $2$-vertex,
	hence there are two possibilities as depicted in Figure~\ref{fig:7face2verts}.
	Moreover, by Lemmas~\ref{lem:separating-cycle} and~\ref{lem:5face4vertex}, 
	$v_4v_5$ is $3$-facially adjacent neither to $u_2v_2$ nor to $u_7v_7$, 
	and so recoloring $u_2v_2$ and/or $u_7v_7$ does not change the set of available colors for $v_4v_5$.
	
	Consider a \fec $\sigma$ of $G/\alpha$ using at most $10$ colors.
	In $G$, $\sigma$ induces a coloring with only the edges of $\alpha$ being non-colored.
	Every non-colored edge incident with a $2$-vertex has at least $5$ available colors 
	and every other edge has at least $4$ available colors.
	Moreover, for every two edges $e_1$ and $e_2$ of $\alpha$ which are both incident with the same $2$-vertex,
	we have that $|A(e_1) \cap A(e_2)| \ge 4$.
	By assumption, there are at least two $2$-vertices in $\alpha$ and thus at least four edges have at least $5$ available colors.	
	This implies that the union of available colors of every subset of $k$ edges is of size at least $k$, for $k \le 5$.
	We divide the proof into three cases regarding the number of colors in the union $A(\alpha)$.
	
	\medskip
	\noindent \textbf{(1)} \quad Suppose first that $|A(\alpha)| = 5$, say $A(\alpha) = [1,5]$. Then $A(v_1v_2) = A(v_1v_7) = A(\alpha)$.
	We may also assume that $\sigma(u_2v_2)=6$ and $\sigma(u_7v_7)=7$.
	We intend to recolor the edges $u_2v_2$ and $u_7v_7$ with two colors $c_1$ and $c_2$ from $A(\alpha)$ 
	such that after recoloring, $c_1$ and $c_2$ will still be available colors for some edges of $\alpha$,
	and so the colors of $[1,7]$ will be available for $E(\alpha)$.
	
	By \cref{lem:3vert-recolor}, $u_2v_2$ can be recolored with at least two colors from $A(v_1v_2) \cap A(v_2v_3)$. 
	Since $|A(v_1v_2) \cap A(v_2v_3) \cap A(v_4v_5)| \ge 4$ in both cases depicted in Figure~\ref{fig:7face2verts}, 
	we can recolor $u_2v_2$ with a color $c_1 \in A(v_1v_2) \cap A(v_2v_3) \cap A(v_4v_5)$, 
	and thus make the color $6$ available for $E(\alpha)$. 
	Next, we recolor $u_7v_7$ with (possibly the only) color from $A(v_1v_7) \cap A(v_6v_7)$, 
	and thus make the color $7$ available for $E(\alpha)$.
	Note that then $c_1$ is still available for $v_4v_5$, $c_2$ is still available for $v_3v_4$, 
	hence all colors of $[1,7]$ are available for $E(\alpha)$.
	
	
	It remains to show that the union of available colors of any six edges of $\alpha$ contains at least $6$ colors.
	Suppose this is not true, 
	and there is $e' \in E(\alpha)$ such that the set $A$ of available colors for $E(\alpha) \setminus \set{e'}$ is of size $5$.
	Then, the set of available colors for $e'$ contains exactly two colors $c_1',c_2' \in [1,7]\setminus A$, 
	and it is easy to see that $\set{c_1',c_2'} = \set{c_1,c_2}$. 
	However, $u_2v_2$ is colored with $c_1 \in A(v_4v_5)$, 
	and $u_7v_7$ is colored with $c_2 \in A(v_3v_4)$, 
	therefore $e'$ can be neither $v_4v_5$ nor $v_3v_4$, a contradiction. 
	Thus, by \cref{thm:Hall}, we can color each non-colored edge with a distinct color from the set $[1,7]$, 
	which provides a \fec of $G$ with at most $10$ colors.
	
	
	\medskip
	\noindent \textbf{(2)} \quad Now, suppose that $|A(\alpha)| = 6$, say $A(\alpha) = [1,6]$.
	First, note that at most one of the edges $u_2v_2$ and $u_7v_7$ is colored with a color from $A(\alpha)$,
	otherwise $|A(\alpha)| \ge |A(v_1v_2)| + 2 \ge 7$. 
	Observe that we can proceed in this way, since the cases {\bf(3)}, when $|A(\alpha)| \ge 7$,
	and {\bf(2)}, when $|A(\alpha)| = 6$, are analyzed independently from each other.
	So, we may assume that $\sigma(u_2v_2) = 7$ or $\sigma(u_7v_7) = 7$.
	We suppose the former, i.e., $\sigma(u_2v_2) = 7$, and note that the proof for the second case proceeds similarly, although not completely symmetrically, 
	due to the assumption that one of the vertices $v_3$ and $v_4$ is a $2$-vertex.
	We consider two cases regarding the color of $u_7v_7$.
	
	\smallskip
	\noindent \textbf{(2.1)} \quad Suppose first that $u_7v_7$ is colored with a color from $A(\alpha)$, say $\sigma(u_7v_7) = 6$.
	Then, $A(v_1v_2) = A(v_1v_7) = A(\alpha) \setminus \set{6} = [1,5]$.	
	We split this case further into two subcases, regarding which of the vertices $v_3$ and $v_4$ is a $2$-vertex
	(recall that, by symmetry, we know precisely one of $v_3$, $v_4$ is of degree $2$).
	
	\smallskip
	\noindent \textbf{(2.1.1)} \quad If $v_3$ is a $2$-vertex, then $|A(v_1v_2) \cap A(v_2v_3)| \ge 4$
		and $|A(v_1v_2) \cap A(v_2v_3) \cap A(v_6v_7)| \ge 3$.
		Thus, by \cref{lem:3vert-recolor}, we can recolor $u_2v_2$ with a color $c_1$ from $A(v_1v_2) \cap A(v_2v_3) \cap A(v_6v_7)$.
		By Lemmas~\ref{lem:separating-cycle} and~\ref{lem:5face4vertex},
		the set of available colors for $v_6v_7$ does not change. 
		Therefore, the set of available colors for $E(\alpha)$ changes to $[1,7]$, 
		and it only remains to show that any set $E \subseteq E(\alpha)$ with $|E| = 6$ 
		has its set of available colors of size at least $6$. 
		So, suppose the contrary, and let $e \in E(\alpha)$ be such that the set of available colors 
		for $E(\alpha)\setminus \set{e}$ is of size $5$. 
		This means that there are two colors in $[1,7]$ that are available only for $e$. 
		Note that all colors of $[1,5] \setminus \set{c_1}$ are available for $v_1v_2$ and $v_1v_7$, 
		and color $7$ is available for $v_1v_2$, $v_2v_3$, $v_3v_4$, and $v_1v_7$. 
		So, the above two colors must be $c_1$ and $6$. 
		However, since $c_1$ is available for $v_6v_7$, while $6$ is not 
		(recall that $u_7v_7$ is colored with $6$), 
		no edge $e \in E(\alpha)$ can have the required property, 
		a contradiction. 
		Hence, we can apply \cref{thm:Hall} to extend the present coloring of $G/\alpha$ to $G$.
		

	\smallskip
	\noindent \textbf{(2.1.2)} \quad If $v_4$ is a $2$-vertex, then $v_3$ is a $3^+$-vertex.
		If there is a color $c_1$ from $A(v_1v_2) \cap A(v_2v_3)$ with which we can recolor $u_2v_2$
		and $c_1$ is also in the set of available colors of some edge that is not 
		$3$-facially adjacent to $u_2v_2$,
		then we proceed as in the case {\bf (2.1.1)}. 
		So we may assume that $u_2v_2$ can only be recolored with a unique color, say, $1$, 
		meaning that $|A(v_1v_2) \cap A(v_2v_3)| = 3$, and, without loss of generality,
		$A(v_2v_3)=\set{1,2,3,6}$, $A(v_3v_4)=A(v_4v_5)=[2,6]$, $A(v_5v_6)\subset [2,6]$, and $A(v_6v_7)=[2,5]$.		
		Now, by \cref{lem:3vert-recolor}, there are at least two colors from $A(v_1v_7) \cap A(v_6v_7) \subseteq A(v_3v_4)$ to recolor $u_7v_7$,
		and we do it with color $c_2$.
		By \cref{lem:2-con}, the edge $u_2v_2$ is incident with two distinct faces; 
		let $\alpha_1$ be that incident with $v_1v_2$, and let $\alpha_2$ be the other one. 
		Next, consider $e_i^j$, the $j$-th edge following $u_2v_2$ 
		in the direction from $v_2$ to $u_2$ in $\alpha_i$, for $i = 1,2$ and $j = 1,2,3$. 
		Note that from $|A(v_2v_3)| = 4$ it follows that $d(u_2) \ge 3$, 
		hence, by \cref{lem:2-con}, $\set{e_1^1,e_1^2,e_1^3} \cap \set{e_2^1,e_2^2,e_2^3} = \emptyset$. 
		As a consequence of $\sigma(u_2v_2) = 7$, $\sigma(u_7v_7) = 6$, and $A(v_1v_2) = [1,5]$, 
		we have $\set{\sigma(e_1^1),\sigma(e_1^2)} \subseteq [8,10]$. 
		Moreover, $6 \in A(v_2v_3)$, and so $6 \notin \set{\sigma(e_2^1),\sigma(e_2^2)}$. 
		Finally, $u_2v_2$ can be recolored neither with $2$ nor with $3$; 
		therefore, $\set{\sigma(e_1^3),\sigma(e_2^3)} = \set{2,3}$. 
		The above reasoning shows that after recoloring $u_7v_7$ with $c_2$, 
		we can recolor $u_2v_2$ with $6$, 
		which transforms the set of available colors for $E(\alpha)$ to $[1,7]$. 
		As in the previous case, it remains to verify that any set $E \subseteq E(\alpha)$ 
		with $|E| = 6$ has its set of available colors of size at least $6$. 
		This is true, since $c_2$ and $6$ are available for $v_3v_4$ and $v_4v_5$, 
		$7$ is available for $v_1v_2$, $v_2v_3$, and $v_1v_7$, 
		and each color of $[1,5] \setminus \set{c_2}$ is available for $v_1v_2$ and $v_1v_7$. 
		Thus, again by \cref{thm:Hall}, we can find a required \fec of $G$.
		
	
	\smallskip
	\noindent \textbf{(2.2)} \quad Now, suppose that $u_7v_7$ is colored with a color not in $A(\alpha)$, say $\sigma(u_7v_7) = 8$.
		We proceed as in the previous cases. If there is a color $c_1$ from $A(v_1v_2)\cap A(v_2v_3)$ with which we can recolor $u_2v_2$
		and $c_1$ appears as an available color of a non-colored edge that is not $3$-facially adjacent to $u_2v_2$, then we are done.
		Otherwise, we may assume that all $4$ colors of $A(v_6v_7)$ are available for the edges $v_3v_4$, $v_4v_5$, and $v_5v_6$.
		Then recoloring $u_7v_7$ with a color from $A(v_1v_7) \cap A(v_6v_7)$ 
		(at least one color for such a recoloring is guaranteed by \cref{lem:3vert-recolor}) 
		increases the size of the set of available colors for $E(\alpha)$ to $7$. 
		Again, every color, that is available for $E(\alpha)$, is available for at least two edges of $\alpha$; 
		thus, we can use \cref{thm:Hall} to obtain a contradiction as above.
		
	
	\medskip
	\noindent \textbf{(3)} \quad Finally, suppose that $|A(\alpha)| \ge 7$. 
	To apply \cref{thm:Hall}, we only need to show that any subset of $E(\alpha)$ of size $6$ 
	has the set of available colors of size at least $6$. 
	If this is not the case, there is a set $E \subseteq E(\alpha)$ of size $6$ 
	such that $|A(E)| = 5$, hence $E(\alpha) \setminus E = \set{e}$ implies $|A(e) \setminus A(E)| \ge 2$. 
	Clearly, we have $|A(v_1v_2) \cap A(v_2v_3) \cap A(e')| \ge 3$ for every $e' \in E$. 
	Pick an edge $e'' \in E$ that is not $3$-facially adjacent to $u_2v_2$. 
	By \cref{lem:3vert-recolor}, at least one color from $A(v_1v_2) \cap A(v_2v_3) \cap A(e'')$ 
	can be used to recolor $u_2v_2$. 
	In this way, the size of the set of colors available for $E$ increases to $6$. 
	Besides that, at least one color of $A(e)\setminus A(E)$ remains available for $e$, 
	and so the set of available colors is of size at least $6$ for any subset of $E(\alpha)$ of size $6$ that contains $e$. 
	Thus we can apply \cref{thm:Hall} again.	
	
\end{proof}

\begin{lemma}
	\label{lem:9face2vertex}
	No $9$-face in $G$ is incident with a $2$-vertex.
\end{lemma}

\begin{proof}
	\begin{figure}[htp!]
		$$
			\includegraphics[scale=\figscale]{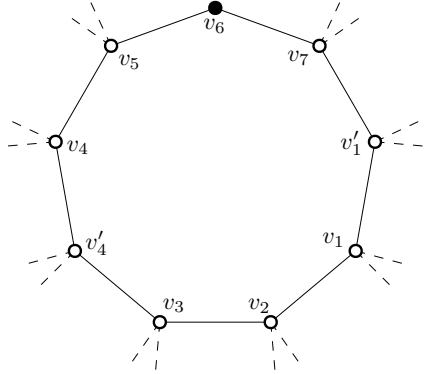}
		$$
		\caption{A reducible $9$-face with an incident $2$-vertex.}
		\label{fig:9face2vert}
	\end{figure}
	Suppose the contrary and let $\alpha$ be a $9$-face incident with a $2$-vertex. 
	We label the vertices as depicted in Figure~\ref{fig:9face2vert}.
	Let $G'$ be the graph obtained by identifying the edges $v_1v_1'$ and $v_4v_4'$, 
	and let $\sigma'$ be a \fec of $G'$ using at most $10$ colors.
	From the coloring of $G$ induced by $\sigma'$ we create the coloring $\sigma$ 
	by uncoloring all edges of $E(\alpha) \setminus \set{v_1v_1',v_4v_4'}$.
	Observe that the edges $v_1v_1'$ and $v_4v_4'$ are not $3$-facially adjacent in $G$, 
	otherwise $G$ would contain a separating cycle of length at most $7$, 
	or a $5$-face with an incident $2$-vertex, contradicting \cref{lem:separating-cycle} or \cref{lem:5face4vertex}.
	Therefore, $\sigma$ is a partial \fec of $G$, in which the edges $v_1v_1'$ and $v_4v_4'$ receive the same color.
	Note that each of the edges $v_1v_2, v_2v_3, v_3v_4', v_4v_5, v_7v_1'$ has at least $3$ available colors, 
	while the two edges $v_5v_6$ and $v_6v_7$, incident with the $2$-vertex $v_6$, have at least $4$ available colors.
	Next, we associate with each edge of $\alpha$ distinct from $v_1v_1'$ and $v_4v_4'$ a variable $X_i$, $i\in \{1,\ldots ,7\}$, in clockwise order starting from $v_1v_2$.
	To apply \cref{thm:nullstellensatz}, we define the following polynomial: 
	\begin{align*}
		F(X_1,\ldots ,X_7) = & (X_1 - X_2)(X_1 - X_3)(X_1 - X_6)(X_1 - X_7)(X_2 - X_3)\\
						   \cdot & (X_2 - X_4)(X_2 - X_7)(X_3 - X_4)(X_3 - X_5)(X_4 - X_5)\\
						   \cdot & (X_4 - X_6)(X_4 - X_7)(X_5 - X_6)(X_5 - X_7)(X_6 - X_7).
	\end{align*}
	The coefficient of the monomial $X_1^2 X_2^2 X_3^2 X_4^2 X_5^2 X_6^3 X_7^2$ in $F(X_1,\ldots X_7)$ is equal to $-3$, 
	thus by \cref{thm:nullstellensatz}, we can extend the coloring $\sigma$ to the \fec of $G$ using at most $10$ colors.
\end{proof}

\begin{lemma}
	\label{lem:10face2vertex}
	Every $10$-face in $G$ is incident with at most two $2$-vertices.
\end{lemma}

\begin{proof}
	Suppose the contrary and let $\alpha$ be a $10$-face in $G$ incident with at least three $2$-vertices.	
	Let the vertices of $\alpha$ be labeled as depicted in Figure~\ref{fig:10face}.
	We prove the lemma by considering three cases regarding the distances between $2$-vertices.
	Namely, it suffices to show that the facial-distance in the face $\alpha$ between two $2$-vertices 
	does not belong to the set $\set{1,3,4}$.
	We do it by using \cref{thm:nullstellensatz}, 
	in which the variable $X_i$ is associated with the edge $v_iv_{i+1}$
	for every $i \in [1,9]$, 
	and the variable $X_{10}$ is associated with the edge $v_1v_{10}$.	
	\begin{figure}[htp!]
		$$
			\includegraphics[scale=\figscale]{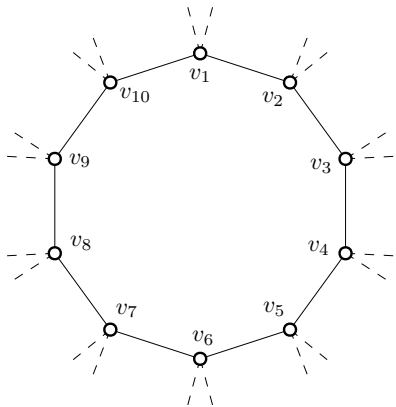}
		$$
		\caption{Labeling of the $10$-face $\alpha$.}
		\label{fig:10face}
	\end{figure}	
	
	\medskip
	\noindent \textbf{(1)} \quad Suppose first that there are two adjacent $2$-vertices in $\alpha$, say $v_1$ and $v_2$.
	Consider the graph $G_1'$ obtained from $G$ by identifying the edges $v_4v_5$ and $v_8v_9$.
	It admits a \fec $\sigma_1'$ using at most $10$ colors.
	The coloring of $G$ induced by $\sigma_1'$ is not necessarily a \fec.
	However, by Lemmas~\ref{lem:separating-cycle} and~\ref{lem:3thread}, $v_4v_5$ and $v_8v_9$ are not $3$-facially adjacent in $G$,
	hence uncoloring the edges of $E(\alpha) \setminus \set {v_4v_5,v_8v_9}$ yields a partial \fec $\sigma_1$ of $G$ with $\sigma_1(v_4v_5) = \sigma_1(v_8v_9)$.
	Note that in this setting, the edges $v_1v_2$, $v_2v_3$, and $v_1v_{10}$ have at least $5$ available colors,
	and the other five edges of $\alpha$ have at least $3$ available colors.
	Now, we define the polynomial:
	\begin{align*}
		F_1(X_1,\ldots,X_{10}) = &(X_1 - X_2)(X_1 - X_3)(X_1 - X_9)(X_1 - X_{10})\\
							\cdot 	&(X_2 - X_3)(X_2 - X_5)(X_2 - X_9)(X_2 - X_{10})\\
							\cdot 	&(X_3 - X_5)(X_3 - X_6)(X_3 - X_{10})(X_5 - X_6)(X_5 - X_7)\\
							\cdot 	&(X_6 - X_7)(X_6 - X_9)(X_7 - X_9)(X_7 - X_{10})(X_9 - X_{10})\,.
	\end{align*}
	Expanding it, we see that the coefficient of the monomial $X_1^4 X_2^4 X_3^2 X_5^2 X_6^1 X_7^2 X_{10}^3$
	in $F_1(X_1,\ldots,X_{10})$ is $1$, and thus, by \cref{thm:nullstellensatz}, 
	we can extend $\sigma_1$ to $G$, a contradiction.
	
	\medskip
	\noindent \textbf{(2)} \quad Suppose now that there are $2$-vertices at distance $3$ in $\alpha$, say $v_1$ and $v_4$.
	Consider the graph $G_2'$ obtained from $G$ by identifying the edges $v_5v_6$ and $v_9v_{10}$.
	It admits a \fec $\sigma_2'$ using at most $10$ colors.
	By Lemmas~\ref{lem:separating-cycle} and~\ref{lem:3thread}, $v_5v_6$ and $v_9v_{10}$ are not $3$-facially adjacent,
	and so by uncoloring the edges of $E(\alpha) \setminus \set{v_5v_6,v_9v_{10}}$,
	we obtain a partial \fec $\sigma_2$ of $G$ with $\sigma_2(v_5v_6) = \sigma_2(v_9v_{10})$.
	Note that in this setting the edges $v_1v_2$, $v_1v_{10}$, $v_3v_4$, and $v_4v_5$ have at least $4$ available colors,
	and the other four edges of $\alpha$ have at least $3$ available colors.
	We define the polynomial:
	\begin{align*}
		F_2(X_1,\ldots ,X_{10}) = &(X_1 - X_2)(X_1 - X_3)(X_1 - X_4)(X_1 - X_8)(X_1 - X_{10})\\
							\cdot 	&(X_2 - X_3)(X_2 - X_4)(X_2 - X_{10})\\
							\cdot 	&(X_3 - X_4)(X_3 - X_6)(X_3 - X_{10})(X_4 - X_6)(X_4 - X_7)\\
							\cdot 	&(X_6 - X_7)(X_6 - X_8)(X_7 - X_8)(X_7 - X_{10})(X_8 - X_{10})\,.
	\end{align*}
	Realizing that the coefficient of the monomial $X_1^3 X_2^2 X_3^2 X_4^3 X_6^2 X_7^2 X_8^1 X_{10}^3$
	in $F_2(X_1,\ldots,X_{10})$ is $-1$,
	we infer that $\sigma_2$ can be extended to $G$ by \cref{thm:nullstellensatz}, a contradiction.
	
	\medskip
	\noindent \textbf{(3)} \quad Suppose now that there are $2$-vertices at distance $4$ in $\alpha$, say $v_1$ and $v_5$.
	Note that the argument of the case {\bf(2)} is also valid here, 
	since the only difference is that the edge $v_3v_4$ may now have only $3$ available colors. 
	This is sufficient for applying \cref{thm:nullstellensatz} to extend $\sigma_2$ to $G$, 
	since the exponent of $X_3$ in the above monomial of $F_2(X_1,\ldots,X_{10})$ is 2.	
	
\end{proof}

\begin{lemma}
	\label{lem:6face7face}
	If a $6$-face $\alpha_1$ and a $7$-face $\alpha_2$ of $G$ share a $2$-vertex $v$,
	and $u \neq v$ is a vertex of $\alpha_2$, then $d(u) \ge 3$.
\end{lemma}

\begin{proof}
	Suppose the contrary and let $u \neq v$ be a $2$-vertex of $\alpha_2$.
	Observe that by Lemma~\ref{lem:6face2thread}, $v$ is the only $2$-vertex incident with both $\alpha_1$ and $\alpha_2$, 
	thus $u$ is either at facial-distance $2$ or at facial-distance $3$ from $v$.
	Consider now the graph $G' = G-v$.
	Note that the remaining edges incident with either $\alpha_1$ or $\alpha_2$ form a $9$-face in $G'$.
	Label the edges of $G$ according to Figure~\ref{fig:9face2vert} with $v_6 = u$, 
	and let $\sigma'$ be a \fec with at most $10$ colors of the graph obtained from $G'$ 
	by identifying the edges $e=v_1v_1'$ and $e'=v_4v_4'$ (as in the proof of \cref{lem:9face2vertex}).
	One can easily observe that in any case, one of the edges $e$ and $e'$ is incident with $\alpha_1$, while the other is incident with $\alpha_2$.
	Thus the edges $e$ and $e'$ are not incident with a common face in $G$. 
	It follows that the only conflict of the coloring $\sigma$ of $G - v$ induced by $\sigma'$ vanishes 
	when the vertex $v$ with its incident edges is added back to $G - v$.
	Finally, since at most $8$ colors appear on the edges incident with $\alpha_1$ and $\alpha_2$, 
	the two non-colored edges incident with $v$ both have at least $2$ available colors.
	Hence, we can extend $\sigma$ to all edges of $G$, a contradiction.
\end{proof}

\begin{lemma}
	\label{lem:7faceSeparatingCycle}
	Let $\alpha_1$ and $\alpha_2$ be distinct $7$-faces of $G$ 
	with a common $2$-vertex $v$ that has a $3$-neighbor $u$ and a $4^+$-neighbor $w$. 
	Furthermore, let $u_1$ and $w_1$ be the vertices of $\alpha_1$ adjacent to $u$ and $w$, respectively. 
	Finally, let either $d(u_1) \ge 3$ and $e_1 \in E(\alpha_1) \setminus \set{uu_1,uv,vw}$ 
	or $d(u_1) = 2$ and $e_1 = ww_1$, 
	and let $e_2 \in E(\alpha_2) \setminus \set{uv,vw}$. 
	Then, the edge $e_1$ is not $3$-facially adjacent to the edge $e_2$.
\end{lemma}

\begin{proof}
	Suppose to the contrary that $e_1$ is within facial-distance $3$ from $e_2$. 
	First realize that the faces $\alpha_1$ and $\alpha_2$ share the vertices $u$, $v$, and $w$ only (use Lemmas~\ref{lem:face-length} and~\ref{lem:5face4vertex}).

	Let $\alpha$ be a face incident with both $\alpha_1$ and $\alpha_2$. 
	Consider a facial path $P$ of length $\ell \le 4$ in $\alpha$ 
	having the first edge $e_1$ and the last edge $e_2$. 
	Note that $e_1 \in E(\alpha_1)\setminus E(\alpha_2)$ 
	and $e_2 \in E(\alpha_2)\setminus E(\alpha_1)$, 
	hence $\alpha$ is unique and $\ell \ge 2$. 
	From $e_1 \neq e_2$ we infer that $|e_1 \cap e_2| \le 1$. 
	Moreover, $|e_1 \cap e_2| = 1$ yields $e_1 \cap e_2 = \set{w}$,
	which in turn means that the edges $e_1$ and $e_2$ are not facially adjacent to each other 
	(since $w$ is a $4^+$-vertex). 
	So, $e_1 \cap e_2 = \emptyset$ and $\ell \in \set{3,4}$.
	
	If $\ell = 3$, then the second edge of $P$ is $x_1x_2$, 
	where $x_1$ is a vertex of $\alpha_1$, $x_2$ is a vertex of $\alpha_2$, 
	and the requirements on $e_1$ imply $u \notin \set{x_1,x_2}$. 
	Moreover, from $d(w) \ge 4$ it follows that $x_1x_2$ is incident neither with $\alpha_1$ nor with $\alpha_2$.
	The faces $\alpha_1$ and $\alpha_2$ create in $G - v$ a $10$-face $\alpha_{1,2}$ 
	incident with ten distinct vertices and ten distinct edges. 
	From the two facial paths joining $x_1$ to $x_2$ in $\alpha_{1,2}$, 
	one does, and the other does not contain the vertex $u$; 
	let $P^+$ be the former and $P^-$ the latter one. 
	Denote by $\ell^+$ and $\ell^-$ the length of $P^+$ and $P^-$, respectively, 
	and so $\ell^+ + \ell^- = 10$. 
	Use $P^+$, $P^-$, and the edge $x_2x_1$ to construct cycles $C^+ = P^+ x_1$ and $C^- = P^- x_1$. 
	The sum of lengths of $C^+$ and $C^-$ is $(\ell^+ + 1) + (\ell^- + 1) = 12$. 
	We have $\min(\ell^+,\ell^-) = \ell^* \le 5$, where $* \in \set{+,-}$.
	If $C^*$ is a separating cycle, its existence contradicts \cref{lem:separating-cycle}. 
	On the other hand, if $C^*$ is not separating, it either bounds a face in $G$ 
	contradicting one of Lemmas~\ref{lem:face-length}, \ref{lem:5face4vertex}, \ref{lem:6face2thread}, 
	or it has at least one chord, which ultimately yields a contradiction to \cref{lem:face-length}.
	
	If $\ell = 4$, then $P = y_1 x_1 z x_2 y_2$, 
	where $e_1 = y_1 x_1$, $e_2 = x_2 y_2$, and $u \notin \set{x_1,x_2}$.
	Let $P^+$, $P^-$, $\ell^+$, $\ell^-$, and $\ell^*$ be defined as in the case $\ell = 3$, 
	and let $C^+ = P^+zx_1$, $C^- = P^-zx_1$. 
	Now, the sum of lengths of $C^+$ and $C^-$ is $(\ell^+ + 2) + (\ell^- + 2) = 14$.
	Again by \cref{lem:separating-cycle}, the cycle $C^*$ is not separating. 
	If $\ell^* \le 4$, a contradiction is reached as above. 
	Finally, assume that $\ell^* = 5 = l^+ = l^-$. 
	Since $C^+$ is not a separating cycle (we can choose $* = +$), 
	and does not have a chord (this would contradict \cref{lem:face-length}), 
	it bounds a face in $G$, which is impossible by Lemmas~\ref{lem:3thread} and~\ref{lem:7face2thread7fac}. 
\end{proof}

\begin{lemma}
	\label{lem:7face7face2verts}
	Let $\alpha_1$ and $\alpha_2$ be two $7$-faces in $G$ that have a common $2$-vertex $v$.
	If $\alpha_1$ and $\alpha_2$ have at least
	 two incident $2$-vertices each, then $v$ has two $4^+$-neighbors.
\end{lemma}

\begin{proof}
	Suppose the contrary and let $u$ and $w$ be distinct neighbors of $v$. 
	By Lemmas~\ref{lem:7face2thread7fac} and~\ref{lem:7face6face}, 
	we may assume, without loss of generality, that $d(u) = 3$ and $d(w) \ge 4$. 
	For both $i \in \set{1,2}$ there is in the face $\alpha_i$ a neighbor $u_i \neq v$ and $w_i \neq v$ of $u$ and $w$, respectively. 
	Again without loss of generality, we may assume that $d(u_1) \ge d(u_2)$. 
	Furthermore, let $\alpha_{1,2}$ be the $10$-face in $G - v$ 
	created from $\alpha_1$ and $\alpha_2$ (cf. the proof of Lemma~\ref{lem:7faceSeparatingCycle}).
	
	\medskip
	\noindent \textbf{(1)} \quad If $d(u_2) \ge 3$, 
	there exist vertices $x \neq v$ and $y \neq v$ incident with $\alpha_1$ and $\alpha_2$, respectively, 
	such that $d(x) = d(y) = 2$. 
	Let $G'$ be the graph obtained from $G - v$ by identifying the two edges 
	incident with $x$ and the two edges incident with $u_2$. 
	By the minimality of $G$, there exists a \fec $\sigma'$ of $G'$ using at most $10$ colors. 
	From the coloring of $G - v$ induced by $\sigma'$ we obtain a coloring $\sigma$ by uncoloring all edges of $\alpha_{1,2}$ 
	that are incident neither with $x$ nor with $u_2$. 
	By \cref{lem:7faceSeparatingCycle}, $\sigma$ is a partial \fec of $G$. 
	Let $E_i$ be the set of (all) three non-colored edges incident in $G$ with the face $\alpha_i$, for $i = 1,2$. 
	We can color the edges of $E_1$ and $E_2$ separately, i.e., 
	when coloring the edges of $E_i$, we suppose that the edges of $E_{3-i}$ are still non-colored, for $i = 1,2$. 
	For that purpose, note that the edge $uu_1 \in E_1$ has at least $3$ available colors, 
	and there is an edge $e_2 \in E_2$ incident with $y$ such that $e_2$ has at least $3$ available colors as well. 
	Furthermore, at least $2$ colors are available for any other edge in $E_1 \cup E_2$. 
	Therefore, by \cref{thm:Hall}, the mentioned separate coloring of edges in $E_1 \cup E_2$ is possible, 
	and, by \cref{lem:7faceSeparatingCycle}, results in a \fec of $G - v$, 
	in which edges incident with $\alpha_{1,2}$ use at most $8$ colors. 
	Two of the remaining colors then suffice to color the edges $uv$ and $vw$.
	
	\medskip
	\noindent \textbf{(2)} \quad  If $d(u_1) \ge 3$ and $d(u_2) = 2$, 
	there exists a $2$-vertex $x$ incident with $\alpha_1$. 
	Let $G'$ be the graph constructed from $G-v$ by identifying the two edges 
	incident with $x$ and the two edges incident with $w_2$. 
	By the minimality of $G$, there exists a \fec $\sigma'$ of $G'$ using at most $10$ colors. 
	From the coloring of $G-v$ induced by $\sigma'$ we obtain a coloring $\sigma$ 
	by uncoloring all edges of $\alpha_{1,2}$ that are incident neither with $x$ nor with $w_2$. 
	By \cref{lem:7faceSeparatingCycle}, $\sigma$ is a partial \fec of $G$. 
	Let the edge sets $E_1$ and $E_2$ be defined as in the case {\bf(1)}. 
	Color first the edges of $E_1$ by \cref{thm:Hall} noting that the number of available colors is at least $4$ 
	for the edge $uu_1$ and at least $2$ for the remaining two edges. 
	Next, color the three edges of $E_2$, again by \cref{thm:Hall}, 
	having in mind that the number of available colors is now at least $3$ for the two edges incident with
	$u_2$ and at least $2$ for the last edge. 
	The edges $uv$ and $vw$ are then colored as before.
	
	\medskip
	\noindent \textbf{(3)} \quad If $d(u_1) = d(u_2) = 2$, 
	let $G'$ be created from $G - v$ by identifying the edges $ww_1$ and $ww_2$. 
	By the minimality of $G$, there exists a \fec $\sigma'$ of $G'$ using at most $10$ colors. 
	From the coloring of $G - v$ induced by $\sigma'$ we obtain a coloring $\sigma$ 
	by uncoloring all edges of $\alpha_{1,2}$ not incident with $w$.
	By \cref{lem:7faceSeparatingCycle}, $\sigma$ is a partial \fec of $G$, 
	and, without loss of generality, we may assume that $\sigma(ww_1) = \sigma(ww_2) = 1$. 
	Denote by $E_i^+$/$E_i^-$ the set of non-colored edges incident in $G$ 
	with the face $\alpha_i$ that are/are not incident with $u_i$, for $i = 1,2$. 
	Notice that the number of available colors is at least $6$ for any edge of $E_1^+ \cup E_2^+$ 
	and at least $3$ for any edge of $E_1^- \cup E_2^-$.

	Suppose now that we are able to use the same color for an edge of $E_1^+$ and an edge of $E_2^-$. 
	Then, by \cref{thm:list-assignment}, $\sigma$ is extendable to $G-v$. 
	A similar extension is possible if the same color can be used either for an edge of $E_1^-$
	and an edge of $E_2^+$, or for the edge of $E_1^-$ 
	incident with $w_1$ and the edge of $E_2^-$ incident with $w_2$. 
	The final extension of the coloring of $G-v$ to $G$ works as in the case {\bf(1)}.	
	
	Thus, we may assume, without loss of generality, 
	that the set of available colors is $[2,4]$ 
	for each edge in $E_1^-$, $[5,10]$ for each edge in $E_2^+$, 
	$[5,7]$ for each edge in $E_2^-$ and $[2,4] \cup [8,10]$ 
	for each edge in $E_1^+$. 
	This, however, leads to a contradiction: 
	since in the facial path $u_2 u u_1 z_1 z_2$ 
	(where $z_2$ is necessarily not incident with $\alpha_1$) 
	the edge $z_1 z_2$ has a color from $[5,7]$, 
	the set of available colors for $uu_2 \in E_2^+$ is not $[5,10]$.
\end{proof}

\begin{lemma}
	\label{lem:7faceThree2verts}
	Let $\alpha_1$ and $\alpha_2$ be two $7$-faces in $G$ that have a common $2$-vertex $v$.
	If $\alpha_1$ has at least three incident $2$-vertices, then $v$ is the only $2$-vertex incident with $\alpha_2$.
\end{lemma}

\begin{proof}
	Suppose the contrary and let $v$ be a $2$-vertex incident with $7$-faces $\alpha_1$ and $\alpha_2$,
	where $n_2(\alpha_1) \ge 3$ and $n_2(\alpha_2) \ge 2$.
	By Lemma~\ref{lem:7face7face2verts}, both neighbors of $v$, $v_1$ and $v_2$, are $4^+$-vertices.
	This implies that every pair of edges $e_1 \in E(\alpha_1)$ and $e_2 \in E(\alpha_2)$, which are not incident with $v$, 
	are not $3$-facially adjacent by Lemmas~\ref{lem:separating-cycle} and~\ref{lem:3thread}.
	Furthermore, by \cref{lem:3thread}, we also have that there exist vertices $u_1$ and $u_2$ of $\alpha_1$ and $\alpha_2$, respectively, 
	such that $u_1,u_2 \notin \set{v,v_1,v_2}$ and $d(u_1),d(u_2) \ge 3$.

	Denote by $\alpha_{1,2}$ the face of the graph $G - v$ created from the faces $\alpha_1$ and $\alpha_2$. 
	Let $G'$ be the graph obtained from $G - v$ by identifying the two edges incident 
	with $u_1$ and the two edges incident with $u_2$. 
	By the minimality of $G$, there exists a \fec $\sigma'$ of $G'$ using at most $10$ colors. 
	From the coloring of $G - v$ induced by $\sigma'$ we obtain a coloring $\sigma$ 
	by uncoloring all edges of $\alpha_{1,2}$ that are incident neither with $u_1$ nor with $u_2$. 
	By mimicking the proof of \cref{lem:7faceSeparatingCycle}, 
	we show that if $e_i$ is any edge of $\alpha_{1,2}$ incident in $G$ with the face $\alpha_i$, for $i = 1,2$, 
	then $e_1$ is not $3$-facially adjacent (in $G$) to $e_2$: 
	since $d(v_1) \ge 4$ and $d(v_2) \ge 4$, 
	the cycle $C^*$ of length at most $7$ from the mentioned proof 
	either is separating or has a chord, in both cases we obtain a contradiction. 
	So, $\sigma$ is a partial \fec of $G - v$.
	
	Let us extend $\sigma$ to a (full) \fec of $G-v$. 
	For that purpose consider in the face $\alpha_{1,2}$ a $2$-vertex $w_1 \neq v$ and a $2$-vertex $w_2 \neq v$ 
	that is in $G$ incident with the face $\alpha_1$ and $\alpha_2$, respectively. 
	If the sets of edges $E_1$ and $E_2$ are defined as in the proof of \cref{lem:7face7face2verts}, case {\bf(1)}, 
	the edges of $E_1$ and those of $E_2$ can be colored separately (using \cref{thm:Hall}). 
	Indeed, $|E_i| = 3$, while the number of available colors is at least $3$ for any (at least one) 
	edge of $E_i$ incident with $w_i$ and at least $2$ for any of the remaining edges of $E_i$, for $i = 1,2$.
	
	The number of available colors is now at least $2$ for both non-colored edges $vv_1$, $vv_2$ of $G - v$, 
	hence there is a \fec of $G$ using at most $10$ colors, a contradiction. 
%
%
\end{proof}

\subsection{Discharging}
\label{sec:disc}

In this part, we decribe the discharging procedure.
First, we assign initial charges to all vertices and faces of $G$.
For every vertex $v \in V(G)$, we set
$$
	\mathrm{ch}_0(v) = 2d(v) - 6\,,
$$
and for every face $\alpha \in F(G)$, we set
$$
	\mathrm{ch}_0(\alpha) = \ell(\alpha) - 6\,.
$$
By Euler's Formula, the total charge of $G$, i.e., the sum of all initial charges, is
\begin{align}
	\label{eq:initch}
	\sum_{v \in V(G)} \mathrm{ch}_0(v) + \sum_{\alpha \in F(G)} \mathrm{ch}_0(\alpha) 
		= \sum_{v \in V(G)} \big(2d(v) - 6 \big) + \sum_{\alpha \in F(G)} \big(\ell(\alpha) - 6 \big) = -12\,.
\end{align}

During the discharging process, we apply the following rules to redistribute the charges between vertices and faces of $G$.
Observe that all amounts of charges sent (and so received, too) by the rules are positive.
\begin{itemize}	
	\item[$R_1~$] Every $4^+$-vertex sends $\frac{1}{5}$ to every incident $5$-face.
	\item[$R_2~$] 
		For each pair $v$ and $u$, 
		where $v$ is a $4^+$-vertex and $u$ is a $2$-vertex adjacent to $v$
		and incident with faces $\alpha_1$ and $\alpha_2$ 
		(note that $\alpha_1 \neq \alpha_2$ by $2$-connectivity of $G$), 
		a charge is sent according to the following
		(without loss of generality, we may assume that $\ell(\alpha_1) \le \ell(\alpha_2)$
		and if $\ell(\alpha_1) = \ell(\alpha_2)$, then $n_2(\alpha_1) \ge n_2(\alpha_2)$):
		\begin{itemize}
			\item[$(a)$] If $\ell(\alpha_1) = 6$, then $v$ sends $\frac{2}{3}$ to $\alpha_1$.
			\item[$(b)$] If $\ell(\alpha_1) = \ell(\alpha_2) = 7$ and $n_2(\alpha_1) = n_2(\alpha_2) = 2$, 
				then $v$ sends $\frac{1}{3}$ to $\alpha_1$ and $\frac{1}{3}$ to $\alpha_2$.
			\item[$(c)$] If $\ell(\alpha_1) = \ell(\alpha_2) = 7$, $n_2(\alpha_1) \ge 2$, and $n_2(\alpha_2) = 1$, 
				then $v$ sends $\frac{2}{3}$ to $\alpha_1$.
			\item[$(d)$] If $\ell(\alpha_1) = 7$ and $\ell(\alpha_2) \ge 8$, 
				then $v$ sends $\frac{2}{3}$ to $\alpha_1$.
		\end{itemize}
	\item[$R_3~$] Every face sends $1$ to every incident $2$-vertex that is not a part of a $2$-thread.
	\item[$R_4~$] Every $7$-face sends $\frac{5}{6}$ to every incident $2$-vertex that is a part of a $2$-thread.
	\item[$R_5~$] Every $8^+$-face sends $\frac{7}{6}$ to every incident $2$-vertex that is a part of a $2$-thread.
\end{itemize}

We prepared all the tools we need to complete our proof of the main theorem.

\begin{proof}[Proof of \cref{thm:main}]
	Clearly, the redistribution of charges does not change the total charge of $G$.
	So,
	\begin{align}
		\label{eq:finalch}
		\sum_{v \in V(G)} \mathrm{ch}_\mathrm{f}(v) + \sum_{\alpha \in F(G)} \mathrm{ch}_\mathrm{f}(\alpha) = -12\,,
	\end{align}
	where $\mathrm{ch}_\mathrm{f}(v)$/$\mathrm{ch}_\mathrm{f}(\alpha)$ 
	stands for the final charge (the ``local'' result of the charge redistribution) 
	of a vertex $v$/a face $\alpha$ of $G$. 
	We are going to show that final charges of vertices and faces of $G$ are all nonnegative. 
	This will mean that the total final charge of $G$ is nonnegative, too, in contradiction to~\eqref{eq:finalch}.

	\medskip
	We first show that each vertex $v \in V(G)$ has a nonnegative final charge.
	In particular, since by \cref{lem:min-deg-2} there are no $1$-vertices in $G$, 
	and $3$-vertices have initial charge $0$ while not sending any charge, 
	we only consider $2$-vertices and $4^+$-vertices.
	
	\textbullet\quad Suppose first that $v$ is a $2$-vertex in $G$, incident with faces $\alpha_1$ and $\alpha_2$.
	Without loss of generality, we assume $\ell(\alpha_1) \le \ell(\alpha_2)$.	
	If $v$ is not a part of a $2$-thread, then it receives $1$ from each of $\alpha_1$ and $\alpha_2$ by $R_3$.
	Hence, $\chf(v) = 2d(v)-6 + 2\cdot 1 = 0$.
	If $v$ is a part of a $2$-thread, then $\ell(\alpha_1) \ge 7$ by \cref{cor:2thread7face}.
	Moreover, by \cref{lem:7face2thread7fac}, $\ell(\alpha_2) \ge 8$,
	and thus by $R_5$, $v$ receives $\frac{7}{6}$ from $\alpha_2$.
	On the other hand, $v$ receives at least $\frac{5}{6}$ from $\alpha_1$ by $R_4$ or $R_5$.
	Hence, $\chf(v) \ge 2d(v)-6 + \frac{5}{6} + \frac{7}{6} = 0$.
	
	\textbullet\quad Now, suppose that $v$ is a $4^+$-vertex. 
	Note that, by \cref{lem:5face4vertex}, $i_5(v) + n_2(v) \le d(v)$, where $i_5(v)$ is the number of $5$-faces incident with $v$.
	Moreover, if $d(v) = 4$, then $n_2(v) \le 3$ by \cref{lem:4vert2verts}, 
	and if additionally $n_2(v)=3$, then $v$ is not incident with a $5$-face by \cref{lem:5face4vertex}.
	Thus, if $d(v) = 4$, then $v$ sends at most $3 \cdot \frac{2}{3}$ of charge by $R_1$ and/or $R_2$, 
	and so $\chf(v) \ge 2d(v) - 6 - 3 \cdot \frac{2}{3} = 0$.
	If $d(v) \ge 5$, then $v$ sends at most $\frac{2}{3}$ of charge for each of at most $d(v)$ adjacent $2$-vertices by $R_1$ and/or $R_2$,
	and so $\chf(v) \ge 2d(v)-6 - d(v) \cdot \frac{2}{3} > 0$.	
	So, after redistribution of charges, all vertices in $G$ have nonnegative final charges.
	
	\medskip
	Next, we show that each face $\alpha \in F(G)$ has a nonnegative final charge. 
	Again, we consider several cases, regarding the length of $\alpha$.
	Recall that by \cref{lem:face-length}, $\alpha$ is of length at least $5$.

	\textbullet\quad Suppose that $\alpha$ is a $5$-face in $G$. 
	By \cref{lem:5face4vertex}, it is incident only with $4^+$-vertices,
	and so it receives $5 \cdot \frac{1}{5}$ by $R_1$.
	Moreover, it does not send any charge, thus $\chf(\alpha) = \ell(\alpha)-6 + 5 \cdot \frac{1}{5} = 0$.
	
	\textbullet\quad Suppose that $\alpha$ is a $6$-face in $G$. 
	By \cref{lem:6face2thread}, every $2$-vertex incident with $\alpha$ is adjacent to two $4^+$-vertices.
	Thus, for every adjacent $2$-vertex, $\alpha$ receives $2 \cdot \frac{2}{3}$ by $R_2(a)$, and sends $1$ by $R_3$.
	Altogether, its final charge is $\chf(\alpha) \ge \ell(\alpha)-6 + 2n_2(\alpha) \cdot \frac{2}{3} - n_2(\alpha) = \frac{1}{3}n_2(\alpha) \ge 0$.
	
	\textbullet\quad Suppose that $\alpha$ is a $7$-face in $G$. 
	It sends charge to incident $2$-vertices by $R_3$ and $R_4$, and it receives charge from incident $4^+$-vertices by $R_2$.	
	We consider the cases regarding incident $2$-vertices.
	If $n_2(\alpha) \le 1$, then, by $R_3$, $\mathrm{ch}_\mathrm{f}(\alpha) \ge \ell(\alpha) - 6 - n_2(\alpha) = 1 - n_2(\alpha) \ge 0$.
	
	Now, suppose that $\alpha$ is incident with two $2$-vertices $v_1$ and $v_2$,
	and let $\alpha_1$ and $\alpha_2$ be the faces incident with $v_1$ and $v_2$, respectively, that are distinct from $\alpha$ (possibly, $\alpha_1 = \alpha_2$).
	Then, by \cref{lem:6face7face}, none of these $2$-vertices is incident with a $6$-face.
	If $v_1$ and $v_2$ form a $2$-thread, then, by \cref{lem:7face2thread7fac}, they are also incident with an $8^+$-face.
	By \cref{lem:7face2th3nei}, at least one of $v_1$ and $v_2$ has a $4^+$-neighbor which sends $\frac{2}{3}$ to $\alpha$ by $R_2(d)$.
	On the other hand, $\alpha$ sends $\frac{5}{6}$ to each of $v_1$ and $v_2$ by $R_4$.
	Hence, $\chf(\alpha) \ge \ell(\alpha)-6 + \frac{2}{3} - 2\cdot\frac{5}{6} = 0$.
	Thus, we may assume that $v_1$ and $v_2$ are not adjacent, 
	and by \cref{lem:7face6face}, each of them has at least one $4^+$-neighbor.
	If $i \in \set{1,2}$ and $\ell(\alpha_i) = 7$, then, by \cref{lem:7faceThree2verts}, 
	$n_2(\alpha_i) \le 2$;
	if, moreover, $n_2(\alpha_i) = 2$, then, by \cref{lem:7face7face2verts},
	$v_i$ has two $4^+$-neighbors.
	Therefore, $\alpha$ receives at least $2\cdot \frac{2}{3}$ by $R_2(b)$, $R_2(c)$, or $R_2(d)$,
	and sends $2\cdot 1$ by $R_3$.
	Hence, $\chf(\alpha) \ge \ell(\alpha)-6 + 2\cdot \frac{2}{3} - 2\cdot 1 = \frac{1}{3}$.
	
	Next, if $\alpha$ is incident with three $2$-vertices, we distinguish two subcases.
	Suppose first that $\alpha$ is incident with a $2$-thread.
	Then, by \cref{lem:7face2thread2vert}, each of the incident $2$-vertices has at least one $4^+$-neighbor,
	and by $R_2(c)$ and $R_2(d)$, $\alpha$ receives at least $3\cdot \frac{2}{3}$ of charge 
	(note that by \cref{lem:7faceThree2verts}, if $\ell(\alpha_1) = 7$, then $\alpha$ receives charge by $R_2(c)$).
	It sends $1$ by $R_3$ and $2\cdot \frac{5}{6}$ by $R_4$.
	Hence, $\chf(\alpha) \ge \ell(\alpha)-6 + 3\cdot \frac{2}{3} - 1 - 2\cdot \frac{5}{6} = \frac{1}{3}$.
	Similarly, if $\alpha$ is not incident with a $2$-thread, 
	then, by \cref{lem:7face6face}, each of the incident $2$-vertices has at least one $4^+$-neighbor,
	and by $R_2(c)$ and $R_2(d)$, $\alpha$ receives at least $3\cdot \frac{2}{3}$ of charge.
	Since $\alpha$ sends $3\cdot 1$ by $R_3$, 
	its final charge is $\chf(\alpha) \ge \ell(\alpha)-6 + 3\cdot \frac{2}{3} - 3\cdot 1 = 0$.
	
	Finally, suppose that $\alpha$ is incident with four $2$-vertices.
	In this case, $\alpha$ is incident with at least one $2$-thread.
	Then, by \cref{lem:7face2thread2vert}, each $2$-vertex incident with $\alpha$ has a $4^+$-neighbor,
	and any $4^+$-vertex incident with $\alpha$ sends $\frac{2}{3}$ to $\alpha$ by $R_2(c)$ or $R_2(d)$ for each of its $2$-neighbors.	
	Since $\alpha$ sends at most $2 \cdot \frac{5}{6}$ and $2 \cdot 1$ by $R_4$ and $R_3$, its final charge is 
	$\chf(\alpha) \ge \ell(\alpha)-6 + 4\cdot\frac{2}{3} - 2 \cdot \frac{5}{6} - 2 \cdot 1 = 0$.
	
	\textbullet\quad By \cref{lem:8face}, we can skip the assumption that $\alpha$ is an $8$-face in $G$.
	
	\textbullet\quad Suppose that $\alpha$ is a $9$-face in $G$. 
	Then, by \cref{lem:9face2vertex}, $\alpha$ is incident with no $2$-vertex, and hence $\chf(\alpha) \ge \ell(\alpha)-6 = 3$.
	
	\textbullet\quad Suppose that $\alpha$ is a $10$-face in $G$. 
	By \cref{lem:10face2vertex}, $\alpha$ is incident with at most two $2$-vertices,
	and so it sends at most $2 \cdot \frac{7}{6}$ charge by $R_3$ or $R_5$.
	So, $\chf(\alpha) \ge \ell(\alpha)- 6 - 2 \cdot \frac{7}{6} = \frac{5}{3}$.
	
	\textbullet\quad Suppose that $\alpha$ is an $11$-face in $G$. 
	Then, by \cref{cor:num2verts}, $\alpha$ is incident with at most five $2$-vertices.
	If $n_2^t(\alpha) = 0$, then it sends charge only by $R_3$. Thus, $\chf(\alpha) \ge \ell(\alpha)-6 - 5 = 0$.
	If $n_2^t(\alpha) \ge 1$, then, by \cref{cor:num2verts}, $n_2(\alpha) \le 4$.
	The charge from $\alpha$ is sent by $R_3$ and/or $R_5$, 
	thus $\chf(\alpha) \ge \ell(\alpha)- 6 - 4\cdot \frac{7}{6} = \frac{1}{3}$.
	
	\textbullet\quad Suppose that $\alpha$ is a $k$-face in $G$, $k \ge 12$.
	If $n_2^t(\alpha) = 0$, then $\alpha$ sends charge only by $R_3$, 
	and so, by \cref{cor:num2verts}, $\chf(\alpha) \ge k-6 - \lfloor k/2 \rfloor \ge \frac{k-12}{2} \ge 0$.
	If $n_2^t(\alpha) > 0$, then $\alpha$ sends charge by $R_3$ and/or $R_5$, 
	in total at most, again by \cref{cor:num2verts}, 
	\begin{align*}
		|S_1(\alpha)| + 2 |S_2(\alpha)| \cdot \frac{7}{6} &\le |S_1(\alpha)| + \frac{7}{3} \cdot \bigg\lfloor \frac{k-2|S_1(\alpha)|}{5} \bigg\rfloor \\
			&\le |S_1(\alpha)| + \frac{1}{15} \big( 7k - 14 |S_1(\alpha)| \big) \\
			&= \frac{1}{15} \big( 7k + |S_1(\alpha)| \big) \\
			&\le \frac{1}{15} \Bigg( 7k + \bigg\lfloor \frac{k}{2} \bigg\rfloor \Bigg) \\
			&\le \frac{k}{2}\,.
	\end{align*}
	Thus, for any value of $n_2^t(\alpha)$, 
	$\chf(\alpha) \ge k - 6 - \frac{k}{2} = \frac{k-12}{2} \ge 0$.
	
	\smallskip
	This proves that every face in $G$ has a nonnegative final charge, which means that the total charge in $G$ is nonnegative,
	which contradicts~\eqref{eq:finalch}.
\end{proof}

\section{Conclusion}

The problems for the edge-coloring version of $\ell$-facial coloring are clearly easier to tackle than those for the vertex version.
Recall that in the vertex version, only the case with $\ell = 1$ is resolved, and moreover, 
its only proof is implied by the Four Color Theorem.
In this paper, we resolved another case and it seems that our approach allows, with some additional effort,
settling the Facial Edge-Coloring Conjecture for several other small values of $\ell$.
However, we failed when trying to generalize our structural lemmas for large values of $\ell$ although faces of lengths at most $\ell+1$ are reducible.
Namely, to apply the discharging method, we need to send enough charge to $2$-vertices of a minimal counterexample $G$, 
and one possibility how to do that is to show that every face of $G$ is incident with at least six $3^+$-vertices.
It turns out that the most problematic faces are those of lengths $k$, for $\frac{3}{2}\ell \le k \le 2\ell$.

Thus, a step towards showing the Facial Edge-Coloring Conjecture would consist 
of finding an efficient approach for resolving the cases with large values of $\ell$.
\begin{problem}
	Find a constant $C$ such that the Facial Edge-Coloring Conjecture holds for every $\ell\ge C$.
\end{problem}

Another line of research consists of determining the upper bounds for the $\ell$-facial chromatic index of plane graphs (and graphs on other surfaces) with additional constraints.
In particular, it remains unknown how high values can the $\ell$-facial chromatic index of a plane graph with minimum degree $3$ achieve.
More generally, further research could be oriented towards
\begin{problem}
	Given $\ell \ge 4$ and $k \in [1,5]$, find an upper bound for $\ell$-facial chromatic index of plane graphs with minimum degree $k$.
\end{problem}




\end{document}